\documentclass[10pt,reqno]{article}

\usepackage{amssymb}
\usepackage{epsfig}
\usepackage{amsmath}
\usepackage{amsthm}
\usepackage{color}
\definecolor{r}{rgb}{0.9,0.3,0.1}
\definecolor{b}{rgb}{0.1,0.3,0.9}

\newtheorem{theorem}{Theorem}[section]
\newtheorem{lemma}[theorem]{Lemma}

\theoremstyle{remark}
\newtheorem{remark}[theorem]{Remark}

\theoremstyle{definition}
\newtheorem{assumption}[theorem]{Assumption}

\newtheorem{definition}[theorem]{Definition}

\newcommand\cbrk{\text{$]$\kern-.15em$]$}}
\newcommand\opar{\text{\,\raise.2ex\hbox{${\scriptstyle
|}$}\kern-.34em$($}}
\newcommand\cpar{\text{$)$\kern-.34em\raise.2ex\hbox{${\scriptstyle |}$}}\,}

\newcommand{\ga}{\gamma}

\newcommand{\om}{\omega}

\newcommand{\Si}{\Sigma}

\newcommand{\de}{\delta}

\newcommand\bL{\mathbb{L}}
\newcommand\bR{\mathbb{R}}
\newcommand\bH{\mathbb{H}}
\newcommand\bZ{\mathbb{Z}}

\newcommand\bE{\mathbb{E}}

\newcommand\cA{\mathcal{A}}
\newcommand\cB{\mathcal{B}}

\newcommand\cF{\mathcal{F}}
\newcommand\cH{\mathcal{H}}
\newcommand\cL{\mathcal{L}}
\newcommand\cP{\mathcal{P}}
\newcommand\cR{\mathcal{R}}

\newcommand\cO{\mathcal{O}}

\newcommand\frH{\mathfrak{H}}

\newcommand{\mysection}[1]{\section{#1}
\setcounter{equation}{0}}

\begin{document}
\setlength{\baselineskip}{16pt}

\title
{A $W^n_2$-Theory of\\
 Stochastic Parabolic Partial Differential Systems\\  on
 $C^1$-domains\\\vspace{0.2cm}
 \Large{(running title: SPDSs on $C^1$-domains)}}

\author{Kyeong-Hun Kim\footnote{Department of
Mathematics, Korea University,  Seoul, South Korea 136-701, \,
kyeonghun@korea.ac.kr. }\quad \hbox{\rm and} \quad Kijung
Lee\footnote{Department of Mathematics, Ajou University, Suwon,
South Korea 443-749, \, kijung@ajou.ac.kr.} }
\date{}


\maketitle

\begin{abstract}
In this article we present a $W^n_2$-theory of stochastic parabolic partial differential systems. In
particular, we focus on non-divergent type. The
space domains we consider are $\bR^d$, $ \bR^d_+$ and  eventually
general bounded $C^1$-domains $\mathcal{O}$. By the nature of
stochastic parabolic equations we need  weighted Sobolev
spaces to prove the existence and the uniqueness. In our choice of
spaces we allow the derivatives of the solution to blow up near the
boundary and moreover the coefficients of the systems are allowed to
oscillate to a great extent or blow up near the boundary.

\vspace*{.125in}


\vspace*{.125in}

\end{abstract}



\mysection{Introduction}

In this article we consider the following general stochastic
parabolic partial differential system :
\begin{eqnarray}
du^k&=&(a^{ij}_{kr}u^r_{x^ix^j}+b^i_{kr}u^r_{x^i}+c_{kr}u^r+f^k)dt\nonumber\\
&&+(\sigma^i_{kr,m}u^r_{x^i}+\nu_{kr,m}u^r+g^k_m)dw^m_t,
\quad t>0,\;x\in \cO\subset \mathbb{R}^d\nonumber\\
u^k(0)&=&u^k_0, \label{eqn main system}
\end{eqnarray}
where $i,j=1,2,\ldots,d$ and $k,r=1,2,\ldots,d_1$ and we used the
summation convention on the repeated indices $i,j,r$. The system
(\ref{eqn main system}) models the interactions among $d_1$
diffusive quantities with other physical phenomena like convection,
internal source or sink, and randomness caused by lack of
information. Moreover, the countable sum of the stochastic integrals
against independent one-dimensional Brownian motions
$\{w^m_{\cdot}:m=1,2,\ldots\}$ enables us to include the stochastic
integral against a cylindrical Brownian motion in (\ref{eqn main
system}) (see sec. 8.2 of \cite{Kr99}). The solution
$u=(u_1,u_2,\cdots,u_{d_1})$ not only depends on $t>0$, $x\in \cO$,
but also depends on $\omega$ in a probability space
$(\Omega,\mathcal{F},\{\mathcal{F}_t;t\ge 0\},P)$ on which
$w^m_{\cdot}$ are defined.  The coefficients $a^{ij}_{kr}, b^i_{kr},
c_{kr},\sigma^i_{kr,m},\nu_{kr,m}$ also depend on $(\omega,t,x)$.
The detailed formulation of (\ref{eqn main system}) follows in the
subsequent sections.

The concrete motivations of studying (\ref{eqn main system}) can be
easily found in the literature. If $d_1=1$,  (\ref{eqn main
system}) is a stochastic partial differential equation(SPDE) of
parabolic type. Such equations arise in many applications of
probability theory (see \cite{Kr99} and \cite{R}). For instance, the
conditional density in nonlinear filtering problems for a partially
observable diffusion process obeys a SPDE and the density of a
super-diffusion process also satisfies a SPDE when the dimension of
the space domain is 1. If $d_1=3$, the motion of a random string can
be modeled by a stochastic parabolic partial differential system
(see \cite{Fu} and \cite{MT}).

General $L_p$-theory with $p\ge 2$ for stochastic parabolic
\emph{equations} (not systems) has been well studied. An
$L_p$-theory of SPDEs with space domain $\bR^n$ was first introduced
by Krylov in \cite{Kr99} (cf. \cite{kr94} for $L_2$-theory), and
since then the results were extended for SPDEs
 defined on  arbitrary $C^1$ domains $\cO$ in $\bR^d$ by Krylov, his
collaborates and many other mathematicians (see, for instance,
\cite{KL1}, \cite{KL2}, \cite{KK}, \cite{Kim03},  \cite{Lo} and
references therein). On the contrary $L_p$-theory of general systems
of type (\ref{eqn main system}) is not available in the literature
except $L_p$-theory of the system with the Laplace operator (see,
for instance, \cite{MR}, \cite{MR04} and the reference therein).

Our goal in this article is to prove unique solvability of the
systems of type (\ref{eqn main system}) in  Sobolev spaces with
weights.
It is known that unless certain compatibility conditions (see, for
instance, \cite{F}) are fulfilled, the second and higher derivatives
of  solutions  blow up near the boundary (see \cite{kr94}). Hence,
we measure this blow-up by using appropriate weights. By the way,
the H\"older space approach does not allow one to obtain results of
reasonable generality (see \cite{KL2} for details).


We extend  the results for single equations in \cite{Kim03},
\cite{KK2}, \cite{Kr99}, \cite{KL1}, and \cite{KL2} to the case of
the systems under the algebraic condition (\ref{assumption 1}) for
the the leading coefficients $a^{ij}_{kr}, \sigma^i_{kr,m}$ and
very minimal smoothness conditions for the coefficients. Under these
assumptions $a^{ij}_{kr}, \sigma^i_{kr,m}$ are allowed to oscillate
to a great extent near the boundary, and $b^i_{kr},c_{kr},
\nu_{kr,m}$ may blow up fast near the boundary. For instance, for
the case $d=d_1=1$ with the space domain $\mathbb{R}_+$ we allow
$a:=a^{11}_{11}$ to behave like $2+\cos|\ln x|^{\alpha}$ near $x=0$,
where $\alpha\in (0,1)$ (see Remark \ref{05.18.01}).
 In this case
the oscillation  of $a(t,x)$ increases to infinity as $x$
approaches the boundary.

For the stability of the numerical solution of (\ref{eqn main
system}),  $W^1_2$-theory may be enough in most cases. But, we are
interested in the regularity of the solutions and we are aiming at
$W^n_p$ theory. \emph{However}, unlike the results for single
equations in \cite{Kim03}, \cite{KK2}, \cite{Kr99}, and \cite{KL2},
we were able to obtain only $W^n_2$-estimates instead of
$W^n_p$-estimates due to many technical difficulties at this point.
For instance, the proofs of  Lemma \ref{a priori 1} and Lemma \ref{a
priori 2} below are not working for $p>2$. Nevertheless, we believe
that $W^n_2$-theory of the system is a main basis for $W^n_p$-theory.
The evidences are the results for single equations. For instance, in
\cite{Kr08} $W^n_p$-theory is established based on
Hardy-Littlewood(HL) theorem, Fefferman-Stein(FS) theorem,
\emph{and} $W^n_2$-theory. In the future we plan to to develop
$W^n_p$-theory of the system (\ref{eqn main system}) by constructing
weighted version of HL and FS theorems and using the result in this
article.

The organization of this article is as follows. Section \ref{Cauchy}
handles the Cauchy problem.  In section  \ref{section half} we prove
the result with space domain $\bR^d_+$ and in section \ref{main
section} we finally prove the results on any bounded $C^1$-domains.

In this article $\bR^{d}$ stands for the Euclidean space of points
\,
$x=(x^{1},...,x^{d})$, \, $\bR^d_+=\{x\in \bR^d:x^1>0\}$ and
$B_r(x):=\{y\in \bR^d:|x-y|<r\}$. For a function $u(x)$ we denote
$$
u_{x^{i}}=\frac{\partial u}{\partial x^{i}}=D_{i}u,\quad
D^{\beta}u=D_{1}^{\beta_{1}}\cdot...\cdot D^{\beta_{d}}_{d}u,
\quad|\beta|=\beta_{1}+...+\beta_{d}
$$
for the multi-indices $\beta=(\beta_{1},...,\beta_{d})$,
$\beta_{i}\in\{0,1,2,...\}$.
 By $c=c(\cdots)$ or $N=N(\cdots)$  we mean that the constant
$c$ or $N$ depends only on what are in parenthesis. Throughout the
article, for functions depending on $\omega,t,x$, the argument
$\omega \in \Omega$ will be omitted.

\mysection{The system with the space domain
$\cO=\bR^d$}\label{Cauchy}

In this section we develop a $W^n_2$-theory of the Cauchy problem
with the system (\ref{eqn main system}). For this we don't need
weights yet since we don't have a boundary.

Let $(\Omega,\mathcal{F},P)$ be a complete probability space and
$\{\mathcal{F}_t:t\ge 0\}$ be a filtration such that $\mathcal{F}_0$
 contains all $P $-null sets of $\Omega $.
By  $\mathcal{P}$ we denote the predictable $\sigma$-algebra on
$\Omega\times(0,\infty)$. Let $\{w^{m}_{t}\}_{m=1}^{\infty}$ be
independent one-dimensional $\{\mathcal{F}_t\}$-adapted Wiener
processes defined on $(\Omega,\mathcal{F},P)$ and
$C^{\infty}_0:=C^{\infty}_0(\mathbb{R}^d;\mathbb{R}^{d_1})$ denote
the set of all $\mathbb{R}^{d_1}$-valued infinitely differentiable
functions with compact support in $\mathbb{R}^d$. By $\mathcal{D}$
we denote the space of $\mathbb{R}^{d}$-valued distributions on
$C^{\infty}_0$; precisely, for $u\in \mathcal{D}$ and $\phi\in
C^{\infty}_0$ we define $(u,\phi)\in \mathbb{R}^{d}$ with components
$(u,\phi)^k=(u^k,\phi^k)$, $k=1,\ldots,d_1$. Here, each $u^k$ is a
usual $\mathbb{R}$-valued distribution defined on
$C^{\infty}(\mathbb{R}^d;\mathbb{R})$.

We define $L_p=L_p(\mathbb{R}^d;\mathbb{R}^{d_1})$ as the space of
all $\mathbb{R}^{d_1}$-valued functions $u=(u^1,\ldots,u^{d_1})$
satisfying
\[
\|u\|^p_{L_p}:=\sum^{d_1}_{k=1}\|u^k\|^p_{L_p(\bR^d)}<\infty.
\]
Let $p \in[2,\infty)$ and $\gamma\in(-\infty,\infty)$. We define the
space of Bessel potential
$H^{\gamma}_p=H^{\gamma}_p(\mathbb{R}^d;\mathbb{R}^{d_1})$ as the
space of all distributions $u$ such that $(1-\Delta)^{\gamma/2}u\in
L_p$, where we define each component of it by
\[
((1-\Delta)^{\gamma/2}u)^k=(1-\Delta)^{\ga/2}u^k
\]
and the operator $(1-\Delta)^{\gamma/2}$ is defined by
\[
(1-\Delta)^{\gamma/2}f=\textrm{ the inverse Fourier transform of
}(1+|\xi|^2)^{\gamma/2}\mathcal{F}(f)(\xi)
\]
with $\cF(f)$ the Fourier transform of $f$. The norm is given by
\[
\|u\|_{H^{\gamma}_p}:=\|(1-\Delta)^{\ga/2}u\|_{L_p}.
\]
Then, $H^{\gamma}_p$ equipped with the given norm is a Banach space
and $C^{\infty}_0$ is dense in $H^{\gamma}_p$ (see \cite{T}).
For non-negative integer $\gamma=0,1,2,\cdots$, it turns out that
$$
H^{\gamma}_p=W^{\gamma}_p:=\{u: D^{\alpha}u\in L_p, \forall \alpha,
|\alpha|\leq \gamma\}.
$$
  It is well known that the first order
differentiation operators,
$\partial_i:H^{\gamma}_{p}(\mathbb{R}^d;\mathbb{R})\to
H^{\gamma-1}_p(\mathbb{R}^d;\mathbb{R})$ given by $u\to u_{x^i}$
$(i=1,2,\ldots,d)$, are bounded. On the other hand, for $u\in
H^{\gamma}_{p}(\mathbb{R}^d;\mathbb{R})$, if $\text{supp}\, (u)
\subset (a,b)\times \mathbb{R}^{d-1}$ with $-\infty<a<b<\infty$, we
have
\begin{equation}
                                        \label{eqn 5.1.1}
\|u\|_{H^{\gamma}_{p}(\mathbb{R}^d;\mathbb{R})}\leq
c(d,\gamma,a,b)\|u_{x}\|_{H^{\gamma-1}_{p}(\mathbb{R}^d;\mathbb{R})}
\end{equation}
(see, for instance, Remark 1.13 in \cite{kr99}).

By $\ell_2$ we denote the set of all real-valued sequences
$e=(e_1,e_2,\ldots)$ with the inner product
$(e,f)_{\ell_2}=\sum_{m=1}^{\infty}e_mf_m$ and the norm
$|e|_{\ell_2}:=(e,e)^{1/2}_{\ell_2}$. If
$g=(g^1,g^2,\cdots,g^{d_1})$ and each $g^k$ is an
 $\ell_2$-valued function,
 then we define
$$
\|g\|^p_{H^{\gamma}_p(\ell_2)}:=\sum_{k=1}^{d_1}\|\;\;|(1-\Delta)^{\ga/2}g^k|_{\ell_2}\;\|^p_{L_p}.
$$
For a fixed time $T<\infty$, we define the stochastic Banach spaces
$$
\bH^{\ga}_p(T)=\bH^{\ga}_p(\bR^d,T):=L_p(\Omega\times(0,T],
\mathcal{P},H^{\ga}_p), \quad
\bH^{\ga}_p(T,\ell_2):=L_p(\Omega\times(0,T],
\mathcal{P},H^{\ga}_p(\ell_2)\;),
$$
$$
\mathbb{L}_p(T):=\bH^0_p(T),\quad
\mathbb{L}_p(T,\ell_2)=\bH^0_p(T,\ell_2)
$$
with the norms given by
\[
\|u\|^p_{\mathbb{H}^{\ga}_p(T)}=\bE\int^{T}_0\|u(t)\|^p_{H^{\gamma}_p}dt,\quad
\|g\|^p_{\mathbb{H}^{\ga}_p(T,\ell_2)}=\bE\int^{T}_0\|g(t)\|^p_{H^{\gamma}_p(\ell_2)}dt.
\]
Finally, we set
$U^{\gamma}_p:=L_p(\Omega,\mathcal{F}_0,H^{\gamma-2/p}_p)$ for the
initial data of the Cauchy problem. The Banach space
$\mathcal{H}^{\ga+2}_p(T)$ below is modified from
$\mathbb{R}$-valued version in \cite{Kr99} to the
$\mathbb{R}^{d_1}$-valued version.

\begin{definition}\label{md}
For a $\mathcal{D}$-valued function
$u=(u^1,\cdots,u^{d_1})\in\mathbb{H}^{\ga+2}_p(T)$, we write
$u\in\mathcal{H}^{\ga+2}_p(T)$ if $u(0,\cdot)\in U^{\gamma+2}_p$,
and there exist $f\in \mathbb{H}^{\ga}_p(T),\;g\in
\mathbb{H}^{\ga+1}_p(T,\ell_2)$ such that, for any $\phi\in
C^{\infty}_0$, (a.s.)  the equality
\begin{equation}\label{e}
(u^k(t,\cdot),\phi)= (u^k(0,\cdot),\phi)+ \int^t_0(
f^k(s,\cdot),\phi)ds+
\sum_{m=1}^{\infty}\int^t_0(g^k_{m}(s,\cdot),\phi)dw^{m}_s
\end{equation}
holds for each $k=1,\cdots,d_1$ and $t\in(0,T]$.  The norm of $u$ in
$\cH^{\gamma+2}_{p}(T)$ is defined by
\[
\|u\|_{\mathcal{H}^{\ga+2}_p(T)}= \|u\|_{\mathbb{H}^{\ga+2}_p(T)}+\|
f\|_{\mathbb{H}^{\ga}_p(T)}+\|g\|_{\mathbb{H}^{\ga+1}_p(T,\ell_2)}+
\|u(0,\cdot)\|_{U^{\ga+2}_p}.
\]
We  write (\ref{e}) in the following simplified ways,
\[
u(t)=u(0)+\int^t_0f(s)ds+ \int^t_0g_m(s)dw^{m}_s\;\textrm{ or  }\;
du=fdt+g_{m}dw^{m}_t,\quad t\in (0,T]
\]
and we say that $du=fdt+g_{m}dw^{m}_t$ holds \emph{in the sense of
distributions}.
\end{definition}


For any $m\times n$ real-valued matrix $C=(c_{kr})$, we define its
norm by
$$
|C|:=\sqrt{\sum_{k=1}^m\sum_{r=1}^n(c_{kr})^2}.
$$
We set $A^{ij}=(a^{ij}_{kr})$, $\Sigma^i=(\sigma^i_{kr})$, and
$\mathcal{A}^{ij}=(\alpha^{ij}_{kr})$, where
$$
\alpha^{ij}_{kr}=\frac{1}{2}\sum_{l=1}^{d_1}
(\sigma^i_{lk},\sigma^j_{lr})_{\ell_2},\quad
\sigma^i_{kr}=(\sigma^i_{kr,1},\sigma^i_{kr,2},\cdots).
$$
 Throughout the
article we assume the followings.

\begin{assumption}
                     \label{main assumptions}
(i) The coefficients $a^{ij}_{kr}, b^i_{kr}, c_{kr},
\sigma^i_{kr,m}$, and $\nu_{kr,m}$ are $\cP \otimes
\cB(\bR^d)$-measurable.

(ii) There exist finite constants $\delta, K^j, L>0$ so that
\begin{equation}\label{assumption 1}
\de|\xi|^2\le\xi^*_{i}\left(A^{ij}-\mathcal{A}^{ij}\right)\xi_{j},
\end{equation}
\begin{equation}\label{assumption 2}
\left|A^{1j}\right|\le K^j,  \quad |\mathcal{A}^{ij}|\le L, \quad
i,j=1,2,\ldots,d
\end{equation}
hold for any $\omega\in\Omega,\ t\ge 0,\ x\in \bR^d$, where $\xi$ is
any (real) $d_1\times d$ matrix, $\xi_i$ is the $i$th column of
$\xi$, $*$ denotes the matrix transpose, and again the summations on $i,j$ are understood.

\end{assumption}

Before we consider the general system (\ref{eqn main system}), we
give a $W^n_2$-theory for the Cauchy problem with the coefficients
independent of $x$:
\begin{equation}
                       \label{eqn system}
du^k=(a^{ij}_{kr}u^r_{x^ix^j}+f^k)dt+(\sigma^{i}_{kr,m}u^r_{x^i}+g^{k}_{m})dw^{m}_t,
\quad u^k(0,\cdot)=u^k_0(\cdot),
\end{equation}
where $i,j=1,2,\cdots,d$,\; $k,r=1,2,\cdots,d_1, \;m=1,2,\ldots$;
recall that we are using summation notation on $i,j,r$.

\begin{theorem}
                      \label{thm 1}
Let $a^{ij}_{kr}=a^{ij}_{kr}(\omega,t)$ and
$\sigma^i_{kr,m}=\sigma^i_{kr,m}(\omega,t)$. Then for any $f\in
\bH^{\gamma}_2(T)$, $g\in \bH^{\gamma+1}_2(T,\ell_2)$, and $u_0\in
U^{\gamma+2}_2$, the problem (\ref{eqn system}) has a unique
solution $u\in \mathcal{H}^{\gamma+2}_2(T)$ and for this solution we
have
\begin{equation}
                        \label{e 6.5.2}
\|u_{xx}\|_{\bH^{\gamma}_{2}(T)}\leq
c\left(\|f\|_{\bH^{\gamma}_2(T)}+\|g\|_{\bH^{\gamma+1}_2(T,\ell_2)}+\|u_0\|_{U^{\gamma+2}_2}\right),
\end{equation}
\begin{equation}
                        \label{e 6.5.3}
\|u\|_{\bH^{\gamma+2}_{2}(T)}\leq
ce^{cT}\left(\|f\|_{\bH^{\gamma}_2(T)}+\|g\|_{\bH^{\gamma+1}_2(T,\ell_2)}+\|u_0\|_{U^{\gamma+2}_2}\right),
\end{equation}
where $c=c(d,d_1,\gamma,\delta,K^j,L)$.
\end{theorem}

\begin{proof}
 Let $\Delta$ denote the usual
Laplace operator. By Theorem 4.10 and Theorem 5.1 in \cite{Kr99},
for each $k$, the single equation
$$
du^k=(\delta \Delta u^k+f^k)dt+g^k_mdw^m_t, \quad u^k(0)=u^k_0
$$
has a solution $u^k\in \mathcal{H}^{\gamma+2}_{2}(T)$. For
$\lambda\in [0,1]$ and $d_1\times d_1$ identity matrix $I$ we define
\begin{eqnarray}
\bar{A}^{ij}_{\lambda}=(\bar{a}^{ij}_{kr,\lambda})
&:=&(1-\lambda)\left(A^{ij}-\mathcal{A}^{ij}\right)+\delta^{ij}\lambda
\delta I\nonumber\\
&=& \left((1-\lambda)A^{ij}+\delta^{ij}\lambda\delta
I\right)-(1-\lambda)\mathcal{A}^{ij}=A^{ij}_{\lambda}-\mathcal{A}^{ij}_{\lambda},\nonumber
\end{eqnarray}
where $A^{ij}_{\lambda}:=(1-\lambda)A^{ij}+\delta^{ij}\lambda\delta
I,\; \mathcal{A}^{ij}_{\lambda}:=(1-\lambda)\mathcal{A}^{ij}$. Then
\begin{eqnarray}
|A^{ij}_{\lambda}|\le |A^{ij}|,\quad |\mathcal{A}^{ij}_{\lambda}|\le
|\mathcal{A}^{ij}|,\quad \delta|\xi|^2\leq
\sum_{i,j}\xi^*_i\bar{A}^{ij}_{\lambda}\xi_j\nonumber
\end{eqnarray}
for any $d_1\times d$-matrix $\xi$. Thus, having the method of
continuity in mind (see the proof of Theorem 5.1 in \cite{Kr99} for
the details), we only prove that the a priori estimates (\ref{e
6.5.3}) and (\ref{e 6.5.2}) hold given that a solution $u$ already
exists.

 {\bf{Step 1}}.
Assume $\gamma=0$. Applying the stochastic product rule
$d|u^k|^2=2u^kdu^k+du^kdu^k$ for each $k$, we have
\begin{eqnarray}
|u^k(t)|^2&=&|u^k_0|^2+\int^t_0
\left[2u^k(a^{ij}_{kr}u^r_{x^ix^j}+f^k)+|\sigma^i_{kr}u^r_{x^i}+g^{k}|_{\ell_2}^2\right]ds\nonumber\\
&&+\int^t_0 2u^k(\sigma^i_{kr,m}u^r_{x^i}+g^{k}_m)dw^{m}_s,\quad
t>0.\label{square}
\end{eqnarray}
Making the summation on $r,i$ appeared, we note that
\begin{eqnarray*}
\sum_k\left|\sum_{r,i}\sigma^i_{kr}u^r_{x^i}+g^{k}\right|_{\ell_2}^2&=&
\sum_k\left[\left|\sum_{r,i}\sigma^i_{kr}u^r_{x^i}\right|_{\ell_2}^2+2(\sum_{r,i}\sigma^i_{kr}u^r_{x^i},g^k)_{\ell_2}+|g^k|^2_{\ell_2}\right]\\
&=&2\sum_{i,j}(u_{x^i})^*\mathcal{A}^{ij}u_{x^j}+2\sum_{k,r,i}(\sigma^i_{kr}u^r_{x^i},g^k)_{\ell_2}+\sum_k|g^k|^2_{\ell_2}.
\end{eqnarray*}
By taking expectation, integrating with respect to $x$, and using
integrating by parts  in order, we get from  (\ref{square})
\begin{eqnarray}
&&\bE\int_{\bR^d}|u(t)|^2dx+2\;\bE\int^t_0\int_{\bR^d}\sum_{i,j}(u_{x^i})^*(A^{ij}-\mathcal{A}^{ij})u_{x^j}dxds\nonumber\\
&=&\bE\int_{\bR^d}|u_0|^2dx
+\bE\int^t_0\int_{\bR^d}\left[2u^*f+2\sum_{k,r,i}(\sigma^i_{kr}u^r_{x^i},g^k)_{\ell_2}+\sum_k|g^{k}|_{\ell_2}^2\right]dxds.
\label{eqn 7.8}
\end{eqnarray}
Note that
\begin{eqnarray}
2\left|\sum_{k,r,i}(\sigma^i_{kr}u^r_{x^i},g^k)_{\ell_2}\right|&\le&
2\sum_{k}\big|\sum_{r,i}\sigma^i_{kr}u^r_{x^i}\big|_{\ell_2}\left|g^k\right|_{\ell_2}\nonumber\\
&\le&\sum_{k}\left(\frac{\varepsilon}{2}
\big|\sum_{r,i}\sigma^i_{kr}u^r_{x^i}\big|_{\ell_2}^2+\frac{2}{\varepsilon}\left|g^k\right|_{\ell_2}^2\right)\nonumber\\
&\le& \frac{\varepsilon}{2}
|u_x|^2\sum_{k,r,i}\big|\sigma^i_{kr}\big|_{\ell_2}^2+\frac{2}{\varepsilon}\sum_k\left|g^k\right|_{\ell_2}^2\nonumber\\
&=& \varepsilon
|u_x|^2\sum_{r,i}\big|\alpha^{ii}_{rr}\big|^2+\frac{2}{\varepsilon}\sum_k\left|g^k\right|_{\ell_2}^2\nonumber\\
\label{2009.06.10 04:25 PM}
\end{eqnarray}
for any $\varepsilon>0$. Hence, it  follows that
\begin{eqnarray}
&&\bE\int_{\bR^d}|u(t)|^2dx+2\delta\;
\bE\int^t_0\int_{\bR^d}|u_x|^2dxds\nonumber\\
&\leq&\bE\int_{\bR^d}|u_{0}|^2dx +\varepsilon\cdot d\cdot
L^2\;\bE\int^t_0\int_{\bR^d}|u_x|^2dxds+ \bE\int^t_0\int_{\bR^d}|u|^2dxds\nonumber\\
&&+\bE\int^t_0\int_{\bR^d}|f|^2dxds+c\;\bE\sum_k\int^t_0\int_{\bR^d}|g^{k}|_{\ell_2}^2dxds.\label{eqn
6.5.5}
\end{eqnarray}
Similarly, for $v=u_{x^n}$ with any $n=1,2,\ldots,d$, we get (see
(\ref{eqn 7.8}))
\begin{eqnarray}
&&\bE\int_{\bR^d}|v(t)|^2dx+2\delta \bE\int^t_0\int_{\bR^d}|v_x|^2dxds\nonumber\\
&=&\bE\int_{\bR^d}|(u_0)_{x^n}|^2dx
+\bE\int^t_0\int_{\bR^d}\left[-2v_{x^n}^*f+2\sum_{k,r,i}(\sigma^i_{kr}v^r_{x^i},g^k_{x^n})_{\ell_2}+\sum_k|g^{k}_{x^n}|_{\ell_2}^2\right]dxds.\nonumber\\
&\leq& \|u_0\|^2_{U^2_2}+\varepsilon
\|u_{xx}\|^2_{\bL_2(t)}+c\|f\|^2_{\bL_2(t)}+c\|g_x\|^2_{\bL_2(t,\ell_2)}.\label{v_x}
\end{eqnarray}
Choosing small $\varepsilon$ and considering all $n$, we have
(\ref{e 6.5.2}). Now, (\ref{v_x}), (\ref{eqn 6.5.5}) and Gronwall's
inequality easily lead to (\ref{e 6.5.3}).

{\bf{Step 2}}. Let $\gamma\neq 0$. The result of this case easily
 follows from the fact that
$(1-\Delta)^{\mu/2}:H^{\gamma}_p\to H^{\gamma-\mu}_p$ is an isometry
for any $\gamma,\mu\in \bR$ when $p\in (1,\infty)$; indeed, $u\in
\cH^{\gamma+2}_2(T)$ is a solution of (\ref{eqn system}) if and only
if $v:=(1-\Delta)^{\gamma/2}u\in \cH^2_2(T)$ is a solution of
(\ref{eqn system}) with $(1-\Delta)^{\gamma/2}f,
(1-\Delta)^{\gamma/2}g, (1-\Delta)^{\gamma/2}u_0$ in places of $f,
g, u_0$ respectively. Moreover, for instance, we have
$$
\|u\|_{\bH^{\gamma+2}_{2}(T)}=\|v\|_{\bH^{2}_{2}(T)}\leq
ce^{cT}\left(\|(1-\Delta)^{\gamma/2}f\|_{\bL_2(T)}+\|(1-\Delta)^{\gamma/2}g\|_{\bH^1_2(T,\ell_2)}
+\|(1-\Delta)^{\gamma/2}u_0\|_{U^2_2}\right)
$$
$$
=ce^{cT}\left(\|f\|_{\bH^{\gamma}_2(T)}+\|g\|_{\bH^{\gamma+1}_2(T,\ell_2)}+\|u_0\|_{U^{\gamma+2}_2}\right).
$$
The theorem is proved.

\end{proof}

Now we extend Theorem \ref{thm 1}  to the case of the Cauchy problem
with variable coefficients. Fix $\varepsilon_0>0$. For $\gamma \in
\bR$ let us define
 $|\gamma|_+=|\gamma|$ if  $|\gamma|=0,1,2,\cdots$ and
 $|\gamma|_+=|\gamma|+\varepsilon_0$ otherwise. Then we define
 $$
 B^{|\gamma|_+}=\begin{cases} B(\bR^d) &: \quad  \gamma=0\\
 C^{|\gamma|-1,1}(\bR^d) &: \quad |\gamma|=1,2,... \\
C^{|\gamma|+\kappa}(\bR^d) &: \quad \text{otherwise},
 \end{cases}
$$
where $B$ is the space of bounded functions,  and $C^{|\gamma|-1,1}$
and $C^{|\gamma|+\kappa}$ are  the usual H\"older spaces.  The
Banach space $B^{|\gamma|_+}$ is also defined for $\ell_2$-valued
functions. For instance,
 if $g=(g_1,g_2,...)$, then $|g|_{B^0}=\sup_x|g(x)|_{\ell_2}$ and
$$
|g|_{C^{n-1,1}}=\sum_{|\alpha|\leq
n-1}|D^{\alpha}g|_{B^0}+\sum_{|\alpha|=n-1}\sup_{x\neq
y}\frac{|D^{\alpha}g(x)-D^{\alpha}g(y)|_{\ell_2}}{|x-y|}.
$$

Here is the main result of this section.
\begin{theorem}
                      \label{thm 2}
Assume that the coefficients $a^{ij}_{kr}, \sigma^i_{kr}$ are
uniformly continuous in $x$, that is, for any $\varepsilon>0$ there
exists $\delta=\delta(\varepsilon)>0$ so that for any $\om$, $t>0$,
$i,j,k,r$,
$$
|a^{ij}_{kr}(\om,t,x)-a^{ij}_{kr}(\om,t,y)|+|\sigma^i_{kr}(\omega,t,x)-\sigma^i_{kr}(\omega,t,y)|_{\ell_2}<\varepsilon,
\quad \text{if}\quad |x-y|<\delta.
$$
Also, assume for any $\om$, $t>0$, $i,j,k,r$,
$$
|a^{ij}_{kr}(\om,t,\cdot)|_{|\gamma|_+}+|b^i_{kr}(\om,t,\cdot)|_{|\gamma|_+}
+|c_{kr}(\om,t,\cdot)|_{|\gamma|_+}
+|\sigma^i_{kr}(\om,t,\cdot)|_{|\gamma+1|_+}
+|\nu_{kr}(\om,t,\cdot)|_{|\gamma+1|_+}  <L.
$$
Then for any $f\in \bH^{\gamma}_2(T)$, $g\in
\bH^{\gamma+1}_2(T,\ell_2)$ and $u_0\in U^{\gamma+2}_2$, the Cauchy
problem (\ref{eqn main system}) has a unique solution $u\in
\mathcal{H}^{\gamma+2}_2(T)$, and for this solution we have
$$
\|u\|_{\bH^{\gamma+2}_{2}(T)}\leq
c\left(\|f\|_{\bH^{\gamma}_2(T)}+\|g\|_{\bH^{\gamma+1}_2(T,\ell_2)}+\|u_0\|_{U^{\gamma+2}_2}\right),
$$
where $c=c(d,d_1,\gamma,\delta,K^j,L,T)$.
\end{theorem}

\begin{proof}
It is enough to repeat the proof of Theorem 5.2 in \cite{Kr99},
where the theorem is proved for single equations. The only
difference is that one needs to use Theorem \ref{thm 1} of this
article, instead of Theorem 4.10 in \cite{Kr99}. We leave the
details to the reader.
\end{proof}

\mysection{The system with the space domain $\cO=\bR^d_+$ }
\label{section half}

In this section we study a $W^n_2$-theory of the initial value
problem with the space domain $\bR^d_+$. We use
the Banach spaces  introduced in \cite{kr99}. Let $\zeta\in C^{\infty}_{0}(\bR_{+})$ be a   function satisfying
\begin{equation}
                                       \label{eqn 5.6.5}
\sum_{n=-\infty}^{\infty}\zeta(e^{n+x})>c>0, \quad \forall x\in \bR,
\end{equation}
where $c$ is a constant. It is each to check that  any nonnegative function
$\zeta$ with the property $\zeta>0$ on $[1,e]$ satisfies (\ref{eqn
5.6.5}). For $\theta,\gamma \in \bR$, let $H^{\gamma}_{p,\theta}$ be
the set of all distributions $u=(u^1,u^2,\cdots u^{d_1})$  on
$\bR^d_+$ such that
\begin{equation}
                                                 \label{100.10.03}
\|u\|_{H^{\gamma}_{p,\theta}}^{p}:= \sum_{n\in\bZ} e^{n\theta}
\|\zeta(\cdot)u(e^{n} \cdot)\|^p_{H^{\gamma}_p} < \infty.
\end{equation}
If $g=(g^1,g^2,\ldots,g^{d_1})$ and each $g^k$ is an
 $\ell_2$-valued function,
 then we define
$$
\|g\|_{H^{\gamma}_{p,\theta}(\ell_2)}^{p}= \sum_{n\in\bZ}
e^{n\theta} \|\zeta(\cdot)g(e^{n} \cdot)\|^p_{H^{\gamma}_p(\ell_2)}.
$$
It is known (\cite{kr99}) that up to equivalent norms the space
$H^{\gamma}_{p,\theta}$ is independent of the choice of $\zeta$.
Also,   for any $\eta\in C^{\infty}_0(\bR_+)$, we have
\begin{equation}
                            \label{eqn 5.6.1}
\sum_{n=-\infty}^{\infty}
e^{n\theta}\|u(e^n\cdot)\eta\|^p_{H^{\ga}_p} \leq c
\sum_{n=-\infty}^{\infty}
e^{n\theta}\|u(e^n\cdot)\zeta\|^p_{H^{\ga}_p},
\end{equation}
where $c$ depends only on $d,\gamma,\theta,p,\eta,\zeta$.
Furthermore,
 if $\gamma$ is a nonnegative integer, then
$$
H^{\gamma}_{p,\theta}=\{u:u,x^1Du,\cdots,(x^1)^{|\alpha|}D^{\alpha}u\in
L_p(\:\bR^d_+,(x^1)^{\theta-d}dx\:),\;|\alpha|\le\gamma\},
$$
$$
\|u\|^p_{H^{\gamma}_{p,\theta}}\sim \sum_{|\alpha|\le\gamma}
\int_{\bR^d_+}|(x^1)^{|\alpha|} D^{\alpha}u(x)|^p(x^1)^{\theta-d}
\,dx.
$$

Below we collect some other properties of spaces
$H^{\gamma}_{p,\theta}$. For $\mu\in\mathbb{R}$ let $M^{\mu}$ be the
operator of multiplying by $(x^1)^{\mu}$ and $M=M^1$.

\begin{lemma} (\cite{kr99})
              \label{lemma 2}
(i) Assume that $\gamma-d/p=m+\nu$ for some $m=0,1,\cdots$ and
$\nu\in (0,1]$.  Then for any $u\in H^{\gamma}_{p,\theta}$ and $i\in
\{0,1,\cdots,m\}$, we have
$$
|M^{i+\theta/p}D^iu|_{C}+[M^{m+\nu+\theta/p}D^m u]_{C^{\nu}}\leq c
\|u\|_{ H^{\gamma}_{p,\theta}}.
$$

(ii) Let $\mu\in \bR$. Then $M^{\mu}H^{\gamma}_{p,\theta+\mu
p}=H^{\gamma}_{p,\theta}$,
$$
\|u\|_{H^{\gamma}_{p,\theta}}\leq c
\|M^{-\mu}u\|_{H^{\gamma}_{p,\theta+\mu p}}\leq
c\|u\|_{H^{\gamma}_{p,\theta}}.
$$

(iii) $M D, DM: H^{\gamma}_{p,\theta}\to H^{\gamma-1}_{p,\theta}$
are bounded linear operators, and it holds that
$$
\|u\|_{H^{\gamma}_{p,\theta}}\leq c\|u\|_{H^{\gamma-1}_{p,\theta}}+c
\|M Du\|_{H^{\gamma-1}_{p,\theta}}\leq c
\|u\|_{H^{\gamma}_{p,\theta}},
$$
$$
\|u\|_{H^{\gamma}_{p,\theta}}\leq c\|u\|_{H^{\gamma-1}_{p,\theta}}+c
\|DM u\|_{H^{\gamma-1}_{p,\theta}}\leq c
\|u\|_{H^{\gamma}_{p,\theta}}.
$$

(iv) The operator $\cL:=M^2\Delta+2MD_1$ is a bounded operator from
$H^{\gamma}_{p,\theta}$ onto $H^{\gamma-2}_{p,\theta}$ with the
bounded inverse $\cL^{-1}$ for any $\gamma$.
\end{lemma}

Let us denote
$$
\bH^{\gamma}_{p,\theta}(T)=L_p(\Omega\times
[0,T],\cP,H^{\gamma}_{p,\theta}), \quad
\bH^{\gamma}_{p,\theta}(T,\ell_2)=L_p(\Omega\times
[0,T],\cP,H^{\gamma}_{p,\theta}(\ell_2)),
$$
$$
U^{\gamma}_{p,\theta}=
M^{1-2/p}L_{p}(\Omega,\cF_0,H^{\gamma-2/p}_{p,\theta}), \quad
\bL_{p,\theta}(T)=\bH^{0}_{p,\theta}(T).
$$
The Banach space $\frH^{\ga+2}_p(T)$ below is modified from
$\mathbb{R}$-valued version in \cite{KL2} to the
$\mathbb{R}^{d_1}$-valued version.
\begin{definition}
We write $u \in \frH^{\gamma+2}_{p,\theta}(T)$ if
 $u=(u^1,\cdots, u^{d_1})\in M \bH^{\gamma+2}_{p,\theta}(T)$,
$u(0,\cdot) \in U^{\gamma+2}_{p,\theta}$, and  for some $f \in
M^{-1}\bH^{\gamma}_{p,\theta}(T)$, $g\in
\bH^{\gamma+1}_{p,\theta}(T,\ell_2)$,
$$
du= f \,dt + g_m \, dw^m_t, \quad t\in [0,T]
$$
in the sense of distributions. We define the norm by
\begin{equation}
                      \label{eqn 11.1}
\|u\|_{\frH^{\gamma+2}_{p,\theta}(T)}=:
\|M^{-1}u\|_{\bH^{\gamma+2}_{p,\theta}(T)} + \|M
f\|_{\bH^{\gamma}_{p,\theta}(T)}  +
\|g\|_{\bH^{\gamma+1}_{p,\theta}(T,\ell_2)} +
\|u(0,\cdot)\|_{U^{\gamma+2}_{p,\theta}}\;\;.
\end{equation}
\end{definition}

\begin{definition}
                    \label{definition good}
Let $A^{ij}=(a^{ij}_{kr})$ and $\Sigma^{i}=(\sigma^i_{kr})$ be {\bf
independent} of $x$. We say that $(A^{ij},\Sigma^i,\theta)$ is
admissible (with constant $N$) if whenever $u\in M\bH^{1}_{2,\theta}(T)$  is a solution of the problem
\begin{eqnarray}
&&du^k=(a^{ij}_{kr}u^r_{x^ix^j}+f^k)dt+(\sigma^{i}_{kr,m}u^r_{x^i}+g^{k}_{m})dw^{m}_t,
\;t>0,\;\;x\in \bR^d_+,\nonumber\\
&& u^k(0,\cdot)=u^k_0(\cdot), \label{2010.10.27.11.27}
\end{eqnarray}
satisfying $u\in L_2(\Omega,C([0,T],C^2_0((1/n,n)\times
\{x':|x'|<n\})))$ for some constant $n>0$, it holds that
\begin{equation}
               \label{eqn main}
\|M^{-1}u\|^2_{\bL_{2,\theta}(T)}\leq
N\left(\|Mf\|^2_{\bL_{2,\theta}(T)}+\|g\|^2_{\bL_{2,\theta}(T,\ell_2)}+\|u_0\|^2_{U^1_{2,\theta}}\right).
\end{equation}

\end{definition}

In  Theorem \ref{sufficient} below we give some sufficient
conditions under which  $(A^{ij},\Sigma^i,\theta)$ is admissible. We
define the symmetric part ($S^{ij}$) and the diagonal part
($S^{ij}_d$) of $A^{ij}$ as follows:
$$
S^{ij}=(s^{ij}_{kr}):=(A^{ij}+(A^{ij})^*)/2, \quad \quad
S^{ij}_d=(s^{ij}_{d,kr}):=(\delta_{kr}a^{ij}_{kr})=(\delta_{kr}s^{ij}_{kr}).
$$
We also define
$$
H^{ij}:=A^{ij}-(A^{ij})^*, \quad
 S^{ij}_o=S^{ij}-S^{ij}_d.
 $$
Assume that there exist constants $\alpha,\beta_1,\cdots,\beta_d\in
[0,\infty)$ such that
\begin{equation}
                         \label{eqn 01.26.1}
  |H^{1j}|\leq \beta^j \quad \forall j=1,2,\ldots,d, \quad\quad |S^{11}_o|\leq
  \alpha.
\end{equation}
We denote
 $$
 K:=\sqrt{\sum_{j=1}^{d} (K^j)^2}, \quad \beta:=\sqrt{\sum_{j=1}^d (\beta^j)^2}.
 $$

\begin{theorem}
                         \label{sufficient}
Let one of the following four conditions be satisfied:
\begin{equation}
                       \label{theta 11}
 \theta\in \left(d-\frac{\delta}{2K-\delta},\,\,
d+\frac{\delta}{2K+\delta}\right),
\end{equation}
\begin{equation}
                   \label{con 22}
\theta\in [d,d+1), \quad 8(d+1-\theta)\delta^2-(\theta-d)\beta^2>0,
\end{equation}
\begin{equation}
             \label{con 11}
\theta\in (d-1,d], \quad
2\delta(d+1-\theta)^2-2(d+1-\theta)(d-\theta)\beta-4(d-\theta)(d+1-\theta)K^1>0,
\end{equation}
\begin{equation}
             \label{con 44}
\theta\in (d-1,d], \quad
\left[\frac{d-\theta}{d+1-\theta}(\beta+2\alpha)+\varepsilon\right]|\xi|^2\leq
\xi^*_{i}\left(A^{ij}-\mathcal{A}^{ij}-2\frac{d-\theta}{d+1-\theta}S^{ij}_d\right)\xi_{j},
\end{equation}
where  $\varepsilon>0$,   $\xi$ is any (real) $d_1\times d$ matrix
and  $\xi_i$ is the $i$th column of $\xi$. Then there exists a
constant $N=N(\theta,\delta,K)>0$ so that $(A^{ij},
\Sigma^i,\theta)$ is admissible with constant $N$.
 \end{theorem}

\begin{remark}

 (i) If $A^{1j}$ are symmetric, i.e., $\beta=0$, then
(\ref{con 11}) combined with (\ref{con 22}) is the same as the
condition $\theta\in (d-\frac{\delta}{2K^1-\delta},d+1)$, which is
weaker than (\ref{theta 11}).

(ii)  If $A^{ij}$ are diagonal matrices and $\Sigma^i=0$, then
$\alpha=\beta^i=0$ and $A^{ij}=S^{ij}_d$. Since
$1-2(d-\theta)/(d+1-\theta) >0$ for $\theta>d-1$, (\ref{con 44})
combined with (\ref{con 22}) is the same as the condition $\theta\in
(d-1,d+1)$. This is the case when the equations in the system is not
correlated.

\end{remark}

\begin{remark}
We do not know how sharp the above conditions are. However, it is
known (\cite{kr99}) that if $\theta\not\in (d-1,d+1)$, then
Theorem \ref{sufficient} is false even for the (deterministic) heat
equation $u_t=\Delta u+f$. i.e., $(\delta^{ij}I,0,\theta)$ is not
admissible for such $\theta$.
\end{remark}

Theorem \ref{sufficient} is proved in the following two lemmas.

\begin{lemma}
                \label{a priori 1}
Assume that $a^{ij}_{kr}, \sigma^i_{kr,m}$ are independent of $x$,
and
\begin{equation}
                     \label{theta 1}
\theta\in \left(d-\frac{\delta}{2K-\delta},\,\,
d+\frac{\delta}{2K+\delta}\right).
\end{equation}
Let $u\in M\bH^{1}_{2,\theta}(T)$ be a solution of
$($\ref{2010.10.27.11.27}$)$ so that $u\in
L_2(\Omega,C([0,T],C^2_0((1/n,n)\times \{x':|x'|<n\})))$ for some
$n>0$. Then we have
\begin{equation}
               \label{eqn main}
\|M^{-1}u\|^2_{\bL_{2,\theta}(T)}\leq
N\left(\|Mf\|^2_{\bL_{2,\theta}(T)}+\|g\|^2_{\bL_{2,\theta}(T,\ell_2)}+\|u_0\|^2_{U^1_{2,\theta}}\right),
\end{equation}
where $N=N(d,d_1,\delta,\theta,K,L)$.
\end{lemma}

\begin{proof}
As in the proof of Theorem \ref{thm 1}, applying the stochastic
product rule  $d|u^k|^2=2u^kdu^k+du^kdu^k$ for each $k$, we have
\begin{eqnarray}
|u^k(t)|^2&=&|u^k_0|^2+\int^t_0
\left[2u^k(a^{ij}_{kr}u^r_{x^ix^j}+f^k)+|\sigma^i_{kr}u^r_{x^i}+g^{k}|_{\ell_2}^2\right]ds\nonumber\\
&&+\int^t_0 2u^k(\sigma^i_{kr,m}u^r_{x^i}+g^{k}_m)dw^{m}_s,\quad
t>0,\nonumber
\end{eqnarray}
where the summations on $i,j,r$ are understood. Denote
$c:=\theta-d$. For each  $k$, we have
\begin{eqnarray}
0&\leq& \bE\int_{\bR^d_+}|u^k(T,x)|^2(x^1)^c
dx\nonumber\\
&=&\bE\int_{\bR^d_+}|u^k(0,x)|^2(x^1)^cdx +
2\bE\int^T_0\int_{\bR^d_+} a^{ij}_{kr}u^ku^r_{x^ix^j}(x^1)^c
dxds\nonumber\\
&&+\bE\int^T_0\int_{\bR^d_+}
|\sigma^i_{kr}u^r_{x^i}|^2_{\ell_2}(x^1)^c
dxds+2\bE\int^T_0\int_{\bR^d_+} (M^{-1}u^k)(Mf^k)(x^1)^c
dxds\label{2009.06.02 04:27 PM}\\
&&\nonumber\\&&+2\bE\int^T_0\int_{\bR^d_+}
(\sigma^i_{kr}u^r_{x^i},g^k)_{\ell_2}(x^1)^c dxds+
\bE\int^T_0\int_{\bR^d_+} |g^k|_{\ell_2}^2(x^1)^c dxds.\nonumber
\end{eqnarray}
Note that, by integration by parts, the second term in the right
hand side of (\ref{2009.06.02 04:27 PM}) is
\begin{eqnarray}
-2\bE\int^T_0\int_{\bR^d_+}a^{ij}_{kr}u^k_{x^i}u^r_{x^j}(x^1)^cdxds-2c\bE\int^T_0\int_{\bR^d_+}(a^{1j}_{kr}u^r_{x^j})(M^{-1}u^k)(x^1)^cdxds.\label{2009.09.02.6:18PM}
\end{eqnarray}
By summing up the terms in the right hand side of (\ref{2009.06.02
04:27 PM}) over $k$ and rearranging the terms, we get
\begin{eqnarray}
&&2\bE\int^T_0\int_{\bR^d_+}u^*_{x^i}\left(A^{ij}-\mathcal{A}^{ij}\right)u_{x^j}\;(x^1)^c
dxds\nonumber\\
&\leq& -2c
\bE\int^T_0\int_{\bR^d_+}a^{1j}_{kr}u^r_{x^j}(M^{-1}u^k)(x^1)^c
\,dxds +\varepsilon\left(\|M^{-1}u\|^2_{\bL_{2,\theta}(T)}
+\|u_x\|^2_{\bL_{2,\theta}(T)}\right)\nonumber\\
&&+c(\varepsilon)\left(\|Mf\|^2_{\bL_{2,\theta}(T)}
+\|g\|^2_{\bL_{2,\theta}(T)}\right)+\|u(0)\|^2_{U^1_{2,\theta}}\nonumber
\\
&\leq& |c|\left(\kappa\|u_x\|^2_{\bL_{2,\theta}(T)}+
K^2\kappa^{-1}\|M^{-1}u\|^2_{\bL_{2,\theta}(T)}\right)
+\varepsilon\left(\|M^{-1}u\|^2_{\bL_{2,\theta}(T)}
+\|u_x\|^2_{\bL_{2,\theta}(T)}\right)\nonumber\\
&&+c(\varepsilon)\left(\|Mf\|^2_{\bL_{2,\theta}(T)}
+\|g\|^2_{\bL_{2,\theta}(T)}\right)
+\|u(0)\|^2_{U^1_{2,\theta}},\label{eqn 22}
\end{eqnarray}
where for the second inequality we used (\ref{assumption 2}),
(\ref{2009.06.10 04:25 PM}), and the fact: for any vectors $v,w\in
\bR^n$ and $\kappa>0$,
$$
|<A^{1j}v,w>|\leq |A^{1j}v||w|\leq K^j|v||w|\leq
\frac12(\kappa|v|^2+\kappa^{-1}(K^j)^2|w|^2);
$$
$\kappa,\varepsilon$ will be decided below. Condition
(\ref{assumption 1}), inequality (\ref{eqn 22}) and the inequality
\begin{equation}
                    \label{eqn 3}
\|M^{-1}u\|^2_{L_{2,\theta}}\leq
\frac{4}{(d+1-\theta)^2}\|u_x\|^2_{L_{2,\theta}}
\end{equation}
(see Corollary 6.2 in \cite{kr99}) lead us to
 \begin{eqnarray}
 &&2\delta\|u_x\|^2_{\bL_{2,\theta}(T)}-
 |c|\left(\kappa
 +\frac{4K^2}{\kappa(d+1-\theta)^2}\right)\|u_x\|^2_{\bL_{2,\theta}(T)}\nonumber\\
&\le&{\varepsilon}\left(\frac{4}{(d+1-\theta)^2}+d^2L\right)\|u_x\|^2_{\bL_{2,\theta}(T)}
+c(\varepsilon)\left(\|Mf\|^2_{\bL_{2,\theta}(T)}+\|g\|^2_{\bL_{2,\theta}(T)}\right)+\|g\|^2_{\bL_{2,\theta}(T)}+\|u(0)\|^2_{U^1_{2,\theta}}.\nonumber
\end{eqnarray}

Now it is enough to  take $\kappa=2K/(d+1-\theta)$ and observe that
(\ref{theta 1}) is equivalent to the condition
$$
2\delta-|c|\left(\kappa
 +\frac{4K}{\kappa(d+1-\theta)^2}\right)=2\delta-\frac{4|c|K}{d+1-\theta}>0.
 $$
 Choosing a small $\varepsilon=\varepsilon(d,d_1,\delta,\theta,K,L)>0$, the lemma is proved.
\end{proof}

\begin{lemma}
                     \label{a priori 2}
 Assume that $a^{ij}_{kr}, \sigma^i_{kr,m}$ are independent of $x$,
 and  one of (\ref{con 22})-(\ref{con 44}) holds. Then the assertion of Lemma \ref{a priori 1} holds.
\end{lemma}
\begin{proof}
We modify the proof of Lemma \ref{a priori 1}.

1. Denote $S^{1j}=(s^{1j}_{kr})=\frac12(A^{1j}+(A^{1j})^*)$ as the
symmetric part of $A^{1j}$. Then $A^{1j}=S^{1j}+\frac12 H^{1j}$ and
for any $\xi\in \bR^{d_1}$
$$
\xi^*A^{1j}\xi=\xi^*S^{1j}\xi.
$$
 Let $c:=\theta-d$. Note that, by
integration by parts, we have
$$
\int_{\bR^d_+}u^*S^{11}u_{x^1}(x^1)^{c-1}dx=
-\frac{c-1}{2}\int_{\bR^d_+}u^* S^{11}u
(x^1)^{c-2}dx=-\frac{c-1}{2}\int_{\bR^d_+}u^* A^{11}u (x^1)^{c-2}dx
$$
and hence
\begin{eqnarray*}
 -2c\int_{\bR^d_+}u^*A^{11}u_{x^1}(x^1)^{c-1}dx&=&-2c\int_{\bR^d_+}u^*S^{11}u_{x^1}(x^1)^{c-1}dx
-c\int_{\bR^d_+}u^*H^{11}u_{x^1}(x^1)^{c-1}dx\\
&=&c(c-1)\int_{\bR^d_+}u^* A^{11}u (x^1)^{c-2}dx
-c\int_{\bR^d_+}u^*H^{11}u_{x^1}(x^1)^{c-1}dx.
 \end{eqnarray*}
Moreover, another usage of integration by parts gives us
\begin{eqnarray}
\int_{\bR^d_+}u^*S^{1j}u_{x^j}(x^1)^{c-1}dx=
-\int_{\bR^d_+}u_{x^j}^* S^{1j}u (x^1)^{c-1}dx=-\int_{\bR^d_+}
u^*(S^{1j})^*u_{x^j} (x^1)^{c-1}dx\nonumber
\end{eqnarray}
for $j\ne 1$, meaning that
$\int_{\bR^d_+}u^*S^{1j}u_{x^j}(x^1)^{c-1}dx=0$ and
\begin{eqnarray}
 -2c\int_{\bR^d_+}u^*A^{1j}u_{x^j}(x^1)^{c-1}dx=-c\int_{\bR^d_+}u^*H^{1j}u_{x^j}(x^1)^{c-1}dx.\nonumber
\end{eqnarray}
Thus the second term in (\ref{2009.09.02.6:18PM}) is
\begin{eqnarray}
&&-2c\bE\int^T_0\int_{\bR^d_+}(a^{1j}_{kr}u^r_{x^j})u^k(x^1)^{c-1}dx\nonumber\\
&=&c(c-1)\bE\int^T_0\int_{\bR^d_+}u^* A^{11}u (x^1)^{c-2}dx
-c\bE\int^T_0\int_{\bR^d_+}u^*H^{1j}u_{x^j}(x^1)^{c-1}dx,\nonumber
\end{eqnarray}
where the summation on $j$ includes $j=1$.

Now, as in the proof of Lemma \ref{a priori 1}, we have
\begin{eqnarray}
&&2\delta\|u_x\|^2_{\bL_{2,\theta}(T)}\nonumber\\
&\le&2\bE\int^T_0\int_{\bR^d_+}u^*_{x^i}\left(A^{ij}-\mathcal{A}^{ij}\right)u_{x^j}\;(x^1)^c
dxds\nonumber\\
&\leq&\bE\int_{\bR^d_+}|u^k(0,x)|^2x^cdx\nonumber\\
&+&c(c-1)\bE\int^T_0\int_{\bR^d_+}a^{11}_{kr}(M^{-1}u^k)(M^{-1}u^r)(x^1)^{c}dxds
-c\bE\int^T_0\int_{\bR^d_+}(h^{1j}_{kr}u^r_{x^j})(M^{-1}u^k) (x^1)^{c} dxds\nonumber\\
&+&2\bE\int^T_0\int_{\bR^d_+} (M^{-1}u^k)(Mf^k)(x^1)^c
dxds+2\bE\int^T_0\int_{\bR^d_+}
(\sigma^{i}_{kr}u^r_{x^i},g^k)_{\ell_2}(x^1)^c dxds\nonumber\\&+&
\bE\int^T_0\int_{\bR^d_+} |g^k|_{\ell_2}^2(x^1)^c
dxds.\label{2009.06.03 06:21 PM}
\end{eqnarray}
Note that  the terms, except the second term and the third term, in
the right hand side of (\ref{2009.06.03 06:21 PM}) are bounded by
\begin{eqnarray}
\varepsilon\left(\|M^{-1}u\|^2_{\bL_{2,\theta}(T)}
+\|u_x\|^2_{\bL_{2,\theta}(T)}\right)+c(\varepsilon)\left(\|Mf\|^2_{\bL_{2,\theta}(T)}
+|g\|^2_{\bL_{2,\theta}(T)}\right)+
\|u(0)\|^2_{U^1_{2,\theta}}.\nonumber
\end{eqnarray}
The second and the third terms   will be estimated below in three steps.

2. If $c(c-1)\geq 0$, hence $\theta\in (d-1,d]$, then we have
\begin{eqnarray*}
&&c(c-1)\bE\int^T_0\int_{\bR^d_+}a^{11}_{kr}(M^{-1}u^k)(M^{-1}u^r)(x^1)^{c}dxds\\
&\leq& c(c-1)K^1\|M^{-1}u\|^2_{\bL_{2,\theta}(T)} \leq
\frac{4}{(d+1-\theta)^2}c(c-1)K^1\|u_x\|^2_{\bL_{2,\theta}(T)}
\end{eqnarray*}
and also
\begin{eqnarray*}
\left|-c\int^T_0\int_{\bR^d_+}(h^{1j}_{kr}u^r_{x^j})(M^{-1}u^k)
(x^1)^{c} dxds\right|&\leq&
\frac{1}{2}|c|\left(\kappa\|u_x\|^2_{\bL_{2,\theta}(T)}
+\kappa^{-1}\beta^2\|M^{-1}u\|^2_{\bL_{2,\theta}(T)}\right)\\
&\leq& \frac{1}{2}|c|\left(\kappa
 +\frac{4\beta^2}{\kappa(d+1-\theta)^2}\right)\|u_x\|^2_{\bL_{2,\theta}(T)}
\end{eqnarray*}
for any $\kappa>0$. To minimize this we take
$\kappa=2\beta/(d+1-\theta)$. Then
\begin{equation}
                      \label{eqn 1.28.1}
\left|-c\int^T_0\int_{\bR^d_+}(h^{1j}_{kr}u^r_{x^j})(M^{-1}u^k)
(x^1)^{c} dxds\right| \leq
\frac{2\beta(d-\theta)}{(d+1-\theta)}\|u_x\|^2_{\bL_{2,\theta}(T)}.
\end{equation}
Thus  from (\ref{2009.06.03 06:21 PM}) we deduce
\begin{eqnarray*}
 &&\left(2\delta-\frac{2\beta(d-\theta)}{(d+1-\theta)}-
 \frac{4}{(d+1-\theta)^2}c(c-1)K^1\right)\|u_x\|^2_{\bL_{2,\theta}(T)}\\
&& \leq \varepsilon \|u_x\|^2_{\bL_{2,\theta}(T)}
+c(\varepsilon)\left(\|Mf\|^2_{\bL_{2,\theta}(T)}+\|g\|^2_{\bL_{2,\theta}(T)}\right)+\|u(0)\|^2_{U^1_{2,\theta}}.
\end{eqnarray*}
This and  (\ref{eqn 3}) yield  the inequality
 (\ref{eqn main}) since (\ref{con 11}) is equivalent to
 $$
 2\delta-\frac{2\beta(d-\theta)}{(d+1-\theta)}-
 \frac{4}{(d+1-\theta)^2}c(c-1)K^1 >0.
 $$

3. Again assume $c(c-1)\geq 0$. By (\ref{2009.06.03 06:21 PM}) and
(\ref{eqn 1.28.1}), we have
\begin{eqnarray*}
&&2\bE\int^T_0\int_{\bR^d_+}u^*_{x^i}\left(A^{ij}-\cA^{ij}\right)u_{x^j}\;(x^1)^c
dxds\\
&\leq&\bE\int_{\bR^d_+}|u^k(0,x)|^2x^cdx\\
&+&c(c-1)\bE\int^T_0\int_{\bR^d_+}\left(s^{11}_{d,kr}+s^{11}_{o,kr}\right)(M^{-1}u^k)(M^{-1}u^r)(x^1)^{c}dxds\\
&+&\frac{2\beta(d-\theta)}{(d+1-\theta)}\|u_x\|^2_{\bL_{2,\theta}(T)}
+\varepsilon\|M^{-1}u\|^2_{\bL_{2,\theta}(T)}+c\|Mf\|^2_{\bL_{2,\theta}(T)}.
\end{eqnarray*}

By Corollary 6.2 of \cite{kr99}, for each $t$, we get
\begin{eqnarray}
&&c(c-1) \int_{\bR^d_+}
s^{11}_{d,kr}(M^{-1}u^k)(M^{-1}u^r)(x^1)^{c}\,dx\nonumber\\
&=&c(c-1)
\int_{\bR^d_+} a^{11}_{kk}|M^{-1}u^k|^2(x^1)^cdx\nonumber\\
&\le&\frac{4(d-\theta)}{(d+1-\theta)}\int_{\bR^d_+}
a^{ij}_{kk}u^k_{x^i}u^k_{x^j}\;(x^1)^c\,dx =
\frac{4(d-\theta)}{(d+1-\theta)}\int_{\bR^d_+}u^*_{x^i}
S^{ij}_{d}u_{x^j}\;(x^1)^c\,dx\nonumber
\end{eqnarray}
and by (\ref{eqn 01.26.1}) and (\ref{eqn 3}),
\begin{eqnarray*}
&&c(c-1)\left|\int_{\bR^d_+} s^{11}_{0,kr}M^{-1}u^k\;M^{-1}u^r(x^1)^c\,dx\right|\nonumber\\
 &\leq& \alpha c(c-1)\int_{\bR^d_+}|M^{-1}u|^2(x^1)^c\,dx
\leq
\frac{4\alpha(d-\theta)}{(d+1-\theta)}\int_{\bR^d_+}|u_x|^2\;(x^1)^c\,dx.
\end{eqnarray*}
It follows that
\begin{eqnarray*}
&&\bE\int^T_0\int_{\bR^d_+}u^*_{x^i}\left(A^{ij}-\cA^{ij}-\frac{2(d-\theta)}{(d+1-\theta)}S^{ij}_d\right)u_{x^j}\;(x^1)^c
dxds\\
&&\leq
\frac{(d-\theta)}{(d+1-\theta)}(\beta+2\alpha)\|u_x\|^2_{\bL_{2,\theta}(T)}
 +\varepsilon \|u_x\|^2_{\bL_{2,\theta}(T)}+
 c(\varepsilon)\left(\|Mf\|^2_{\bL_{2,\theta}(T)}+\|g\|^2_{\bL_{2,\theta}(T)}\right)+\|u(0)\|^2_{U^1_{2,\theta}}.
\end{eqnarray*}
This, (\ref{con 44}) and (\ref{eqn 3}) lead to (\ref{eqn main}).

4. If $c(c-1)\leq 0$, hence $\theta\in [d,d+1)$, then
$$
c(c-1)\bE\int^T_0\int_{\bR^d_+}a^{11}_{kr}(M^{-1}u^k)(M^{-1}u^r)(x^1)^{c}dxds\le
\delta c(c-1)\|M^{-1}u\|^2_{\bL_{2,\theta}(T)};
$$
for this we consider a $d_1\times d$ matrix $\xi$ consisting of
$M^{-1}u$ as the first column and zeros for the rest and apply the
condition (\ref{assumption 1}). Next, as before, we have
\begin{eqnarray*}
\left|-c\bE\int^T_0\int_{\bR^d_+}(h^{1j}_{kr}u^r_{x^j})(M^{-1}u^k)
(x^1)^{c} dxds\right|&\leq&
\frac{1}{2}c\left(\kappa\|u_x\|^2_{\bL_{2,\theta}(T)}
+\kappa^{-1}\beta^2\|M^{-1}u\|^2_{\bL_{2,\theta}(T)}\right)
\end{eqnarray*}
and hence
\begin{eqnarray}
 &&2\delta\|u_x\|^2_{\bL_{2,\theta}(T)}-
 \frac12 c\;\left(\kappa\|u_x\|^2_{\bL_{2,\theta}(T)}
+\kappa^{-1}\beta^2\|M^{-1}u\|^2_{\bL_{2,\theta}(T)}\right)-\delta c(c-1)\|M^{-1}u\|^2_{\bL_{2,\theta}(T)}\nonumber\\
&\le&{\varepsilon}\left(\frac{4}{(d+1-\theta)^2}+d^2L\right)\|u_x\|^2_{\bL_{2,\theta}(T)}
+c(\varepsilon)\left(\|Mf\|^2_{\bL_{2,\theta}(T)}+K\|g\|^2_{\bL_{2,\theta}(T)}\right)\nonumber\\
&&+\|g\|^2_{\bL_{2,\theta}(T)}+\|u(0)\|^2_{U^1_{2,\theta}}.\label{2009.06.03
07:56 PM}
\end{eqnarray}
As we take
$$
\kappa=\frac{\beta^2}{2\delta(1-c)},$$ the terms with
$\|M^{-1}u\|^2_{\bL_{2,\theta}(T)}$ in the left hand side of
(\ref{2009.06.03 07:56 PM}) are canceled out. Now, (\ref{con 22})
which is equivalent to $2\delta-\frac{c\beta^2 }{4\delta(1-c)}>0$
gives us (\ref{eqn main}). The lemma is proved.
\end{proof}

The following lemma with Definition \ref{definition good} will lead
to an a priori estimate:
\begin{lemma}
                   \label{lemma 11.1}
Let $\mu\in \bR$, $f\in M^{-1}\bH^{\mu}_{2,\theta}(T),\;g\in
\bH^{\mu+1}_{2,\theta}(T,\ell_2),\; u(0)\in
U^{\mu+2}_{2,\theta}$ and $u\in M\bH^{\mu+1}_{2,\theta}(T)$ be
a solution of  the problem (\ref{2010.10.27.11.27}) on $[0,T]\times
\mathbb{R}^d_+$, then $u\in M\bH^{\mu+2}_{2,\theta}(T)$ and
\begin{equation}
                              \label{eqn 5.1}
\|M^{-1}u\|_{\bH^{\mu+2}_{2,\theta}(T)}\leq
c\left(\|M^{-1}u\|_{\bH^{\mu+1}_{2,\theta}(T)}+\|Mf\|_{\bH^{\mu}_{2,\theta}(T)}
+\|g\|_{\bH^{\mu+1}_{2,\theta}(T,\ell_2)}+\|u(0)\|_{U^{\mu+2}_{2,\theta}}\right),
\end{equation}
where $c=c(d,d_1,\mu,\theta,\delta,K,L)$.
\end{lemma}
\begin{proof}
By Lemma \ref{lemma 2} (ii) and (\ref{eqn 5.1.1}), we have
\begin{eqnarray*}
 \|M^{-1}u\|^2_{\bH^{\mu+2}_{2,\theta}(T)}
&\leq&
c\sum_{n}e^{n(\theta-2)}\|u(t,e^nx)\zeta(x)\|^2_{\bH^{\mu+2}_2(T)}\\
&=&
c\sum_{n}e^{n\theta}\|u(e^{2n}t,e^nx)\zeta(x)\|^2_{\bH^{\mu+2}_2(e^{-2n}T)}\\
&\leq&
c\sum_{n}e^{n\theta}\|(u(e^{2n}t,e^nx)\zeta(x))_{xx}\|^2_{\bH^{\mu}_2(e^{-2n}T)}.
\end{eqnarray*}
Denote
\begin{eqnarray}
v_n(\omega,t,x)=u(\omega,e^{2n}t,e^nx)\zeta(x),\quad
(a_n)^{ij}_{kr}(\omega,t)=a^{ij}_{kr}(\omega,e^{2n}t),\quad
(\sigma_n)^i_{kr}(\omega,t)=\sigma^i_{kr}(\omega,e^{2n}t)\nonumber
\end{eqnarray}
\begin{eqnarray}
A^{ij}_n=((a_n)^{ij}_{kr}),\quad
\Sigma^i_n=((\sigma_n)^i_{kr}).\nonumber
\end{eqnarray}
Then, since $v_n$ has compact support in $\bR^d_+$, we can regard it
as a distribution defined on the whole space. Thus $v_n$ is in
$\bH^{\mu+1}_2(e^{-2n}T)$ and satisfies
$$
dv_n=\left(A^{ij}_n(v_n)_{x^ix^j}+f_n\right)dt+((\Sigma^i_n)_m(v_n)_{x^i}+(g_n)_m)d(e^{-n}w^m_{e^{2n}t}),
\quad v_n(0,x)=\zeta(x)u_0(e^nx),
$$
where $(\Sigma^i_n)_m=((\sigma_n)^i_{kr,m})$ and
$$
f_n=-2e^nA^{ij}_n u_{x^i}(e^{2n}t,e^nx)\zeta_{x^j}(x) -A^{ij}_n
u(e^{2n}t,e^nx)\zeta_{x^ix^j}(x)+e^{2n}f(e^{2n}t,e^nx)\zeta(x),
$$
$$
(g_n)_m=-(\Sigma^i_n)_mu(e^{2n}t,e^nx)\zeta_{x^i}(x)+e^ng_m(e^{2n}t,e^nx)\zeta(x).
$$
Then, by Theorem \ref{thm 1}, $v_n$ is in
$\mathcal{H}^{\mu+2}_2(e^{-2n}T)$ and
$$
\|(v_{n})_{xx}\|^2_{\bH^{\mu}_2(e^{-2n}T)}\leq
c(d,d_1,\mu,\delta,K,L)\left(\|f_n\|^2_{\bH^{\mu}_2(e^{-2n}T)}+\|g_n\|^2_{\bH^{\mu+1}_2(e^{-2n}T,\ell_2)}
+\|\zeta(x)u_0(e^nx)\|^2_{U^{\mu+2}_2}\right).
$$
 Thus, by (\ref{eqn 5.6.1})
and  Lemma \ref{lemma 2},
\begin{eqnarray*}
&&\sum_{n}e^{n\theta}\|(u(e^{2n}t,e^nx)\zeta(x))_{xx}\|^2_{\bH^{\mu}_2(e^{-2n}T)}\\
&\leq&
c\sum_{n}\left[e^{n\theta}\|u_{x}(t,e^n\cdot)\zeta_{x}\|^2_{\bH^{\mu}_2(T)}\right]
+c\sum_{n}e^{n(\theta-2)}\|u(t,e^n\cdot)\zeta_{xx}\|^2_{\bH^{\mu}_2(T)}\\&&+
c\sum_{n}e^{n(\theta+2)}\|f(t,e^n\cdot)\zeta\|^2_{\bH^{\mu}_2(T)}
+c\sum_{n}\left[e^{n(\theta-2)}\|u(t,e^n\cdot)\zeta_{x}\|^2_{\bH^{\mu}_2(T)}\right]\\&&
+c\sum_{n}e^{n\theta}\|g(t,e^nx)\zeta\|^2_{\bH^{\mu}_2(T,\ell_2)}+\sum_{n}e^{n\theta}\|u_0(t,e^nx)\zeta\|^2_{U^{\mu+2}_2}\\
&\leq &c\|M^{-1}u\|^2_{\bH^{\mu+1}_{2,\theta}(T)}+c\|M
f\|^2_{\bH^{\mu}_{2,\theta}(T)}+c\|g\|^2_{\bH^{\mu+1}_{2,\theta}(T,\ell_2)}+c\|u_0\|^2_{U^{\mu+2}_{2,\theta}}.
\end{eqnarray*}
The lemma is proved.
\end{proof}

From this point on we assume the following:
\begin{assumption}
                   \label{assump 9.17}
There exists a constant $N>0$, {\bf{independent of $x$}},  so that for each
fixed $x$, $(A^{ij}(\cdot,\cdot,x),\Sigma^i(\cdot,\cdot,x),\theta)$
is admissible with constant $N$.
\end{assumption}

First, we prove our results for the problem (\ref{eqn main}) with
the coefficients independent of $x$.
\begin{theorem}
                          \label{theorem half-constant}
Suppose  Assumptions \ref{main assumptions} and  \ref{assump 9.17}
hold. Also assume that $A^{ij}, \Sigma^i$ are independent of $x$.
Then for any $f\in M^{-1}\bH^{\gamma}_{2,\theta}(T),\;g\in
\bH^{\gamma+1}_{2,\theta}(T,\ell_2),\; u_0\in
U^{\gamma+2}_{2,\theta}$, the problem (\ref{2010.10.27.11.27})
admits a unique solution $u\in \frH^{\gamma+2}_{2,\theta}(T)$, and
for this solution
\begin{equation}
                     \label{eqn 11.2}
\|u\|_{\frH^{\gamma+2}_{2,\theta}(T)}\leq
c\left(\|Mf\|_{\bH^{\gamma}_{2,\theta}(T)}+\|g\|_{\bH^{\gamma+1}_{2,\theta}(T,\ell_2)}+\|u_0\|_{U^{\gamma+2}_{2,\theta}}\right),
\end{equation}
where $c=c(d,d_1,\delta,\theta,K^j,L)$.
\end{theorem}

\begin{proof}
 1. By Theorem 3.3 in \cite{KL2}, for each $k$, the single equation
$$
du^k=(\delta \Delta u^k+f^k)dt+g^kdw_t, \quad u^k(0)=u^k_0
$$
has a solution $u^k\in \frH^{\gamma+2}_{2,\theta}(T)$. As in the
proof of Theorem \ref{thm 1} we only need to show that the estimate
(\ref{eqn 11.2}) holds given that a solution already exists. Also by Lemma \ref{lemma 2} and  (\ref{eqn 11.1}) it is enough to show
\begin{equation}
                        \label{a priori 100}
\|M^{-1}u\|_{\bH^{\gamma+2}_{2,\theta}(T)}\leq
c\left(\|Mf\|_{\bH^{\gamma}_{2,\theta}(T)}+\|g\|_{\bH^{\gamma+1}_{2,\theta}(T,\ell_2)}+\|u_0\|_{U^{\gamma+2}_{2,\theta}}\right).
\end{equation}

2. Assume $\gamma \geq 0$. By Theorem 2.9 in \cite{KL2}, for any nonnegative integer $n\geq
\gamma$, the set
$$
\frH^{n}_{2,\theta}(T) \cap
\bigcup_{N=1}^{\infty}L_2(\Omega,C([0,T],C^n_0((1/N,N)\times
\{x':|x'|<N\})))
$$
is dense in $\frH^{\gamma}_{2,\theta}(T)$ and we may assume that $u$
is sufficiently smooth in $x$ and vanishes near the boundary. Let $m$ be an integer so that $\gamma+1-m\leq 0$. Then by applying  Lemma \ref{lemma 11.1} with $\mu=\gamma,\gamma-1,\cdots,\gamma-m$ in order,
$$
\|M^{-1}u\|_{\bH^{\gamma+2}_{2,\theta}(T)}\leq
c\left(\|M^{-1}u\|_{\bH^{\gamma+1-m}_{2,\theta}(T)}+
\|Mf\|_{\bH^{\gamma}_{2,\theta}(T)}+\|g\|_{\bH^{\gamma+1}_{2,\theta}(T,\ell_2)}+\|u_0\|_{U^{\gamma+2}_{2,\theta}}\right).
$$
Thus to get (\ref{a priori 100}) it is enough to use the fact $\|\cdot\|_{H^{\gamma+1-m}_{2,\theta}}\leq \|\cdot\|_{L_{2,\theta}}$ and the inequality (\ref{eqn main}).

3. Assume $\gamma\in [-1,0)$, i.e., $\gamma+1\geq 0$. Recall
$\mathcal{L}u:=(M^2\Delta+2MD_1)u=(x^1)^2\Delta u+2x^1 u_{x^1}$. We
have $\bar{f}:=\cL^{-1}f\in M^{-1}\bH^{\gamma+2}_{2,\theta}(T)$,
$\bar{g}:=\cL^{-1}g\in \bH^{\gamma+3}_{2,\theta}(T,\ell_2)$ and
$\bar{u}_0:=\cL^{-1}u_0\in U^{\gamma+4}_{2,\theta}$. If $\bar{u}\in
\frH^{\gamma+4}_{2,\theta}(T)$ is the solution of the problem
$$
d\bar u=(A^{ij}\bar u_{x^ix^j}+\bar f)dt+(\Si^i_m \bar u_{x^i}+\bar
g_m)dw^m_t,\quad \bar u(0)=\bar u_0
$$
with $\Sigma^i_m=(\sigma^i_{kr,m})$, then for $v=\cL\bar{u}$ we have
$v\in \frH^{\gamma+2}_{2,\theta}(T)$ and
\begin{eqnarray*}
dv&=&\left(A^{ij} v_{x^ix^j}+f-2(A^{1i}+A^{i1})(\bar
u_{x^1x^i}+x^1\Delta \bar{u}_{x^i})-2A^{11}\Delta\bar
u\right)dt\nonumber\\&&+\left(\Si^i_m v_{x^i}+g_m-2\Si^1_m(\bar
u_{x^1}+x^1 \Delta \bar{u})\right)dw^m_t,\quad
t>0\nonumber\\v(0)&=&u_0.
\end{eqnarray*}
Since $\bar u_{x^1x^i}+x^1\Delta
\bar{u}_{x^i}=M^{-1}\mathcal{L}(\bar{u}_{x^i})\in
M^{-1}\bH^{\gamma+1}_{2,\theta}(T)$, $\bar{u}_{x^ix^j}\in
M^{-1}\bH^{\gamma+2}_{2,\theta}(T)$, $\bar{u}_{x^1}\in
\bH^{\gamma+3}_{2,\theta}(T)$, and $\gamma+1\geq 0$, we can find a
$\tilde{u}\in \frH^{\gamma+3}_{2,\theta}(T)$ as the solution of
\begin{eqnarray*}
d\tilde{u}&=&\left(A^{ij} \tilde{u}_{x^ix^j}-2(A^{1i}+A^{i1})(\bar
u_{x^1x^i}+x^1\Delta \bar{u}_{x^i})-2A^{11}\Delta\bar
u\right)dt\nonumber\\&&+\left(\Si^i_m \tilde{u}_{x^i}-2\Si^1_m(\bar
u_{x^1}+x^1 \Delta
\bar{u})\right)dw^m_t,\nonumber\\\tilde{u}(0)&=&0.
\end{eqnarray*}
Then $u:=v-\tilde{u}\in \frH^{\gamma+2}_{2,\theta}(T)$ satisfies
(\ref{2010.10.27.11.27}) and estimate (\ref{a priori 100}) follows
from the formula defining $v, \tilde{u}$ and the fact that
$$
\|M^{-1}v\|_{\bH^{\gamma+2}_{2,\theta}(T)}\leq c
\|M^{-1}\bar{u}\|_{\bH^{\gamma+4}_{2,\theta}(T)}, \quad
\|M^{-1}\tilde{u}\|_{\bH^{\gamma+2}_{2,\theta}(T)}\leq c
\|M^{-1}\bar{u}\|_{\bH^{\gamma+4}_{2,\theta}(T)}.
$$
Now, we pass to proving the uniqueness of the solution in the space
$\frH^{\gamma+2}_{2,\theta}(T)$. Let
$u\in\frH^{\gamma+2}_{2,\theta}(T)$ be a solution with $f=0,\;g=0,
\;u_0=0$. We claim that $u\equiv 0$. For this we just show that
$u\in \frH^{\gamma+3}_{2,\theta}(T)$, or equivalently
$v=\cL^{-1}u\in \frH^{\gamma+5}_{2,\theta}(T)$ since we have already
proved the uniqueness in $\frH^{\gamma+3}_{2,\theta}(T)$ at step 2;
recall $\gamma+1\ge 0$. In fact, since $u\in
\frH^{\gamma+2}_{2,\theta}(T)$ at least, we have $v\in
\frH^{\gamma+4}_{2,\theta}(T)$ and
$$
dv=(A^{ij}v_{x^ix^j}+\bar f)dt+(\Si^i_m v_{x^i}+\bar g_m)dw^m_t,
$$
where
$$
\bar{f}=A^{ij}\cL^{-1}(u_{x^ix^j})-A^{ij}(\cL^{-1}u)_{x^ix^j},\quad
\bar{g}_m=\Si^i_m \cL^{-1}(u_{x^i})-\Si^i_m(\cL^{-1}u)_{x^i}.
$$
However, we observe
\begin{eqnarray}
\cL
\bar{f}&=&A^{ij}(u_{x^ix^j}-\cL((\cL^{-1}u)_{x^ix^j}))\nonumber\\&=&2(A^{1i}+A^{i1})M^{-1}u_{x^i}-8A^{11}M^{-2}u+2A^{11}\Delta(\cL^{-1}u)\in
M^{-1}\bH^{\gamma+1}_{2,\theta}(T),\nonumber\\
\cL \bar{g}&=&\Si^i(u_{x^i}-\cL((\cL^{-1}u)_{x^i}))=2\Si^1M^{-1}u\in
\bH^{\gamma+2}_{2,\theta}(T,\ell_2).
\end{eqnarray}
Thus $\bar{f}\in M^{-1}\bH^{\gamma+3}_{2,\theta}(T)$ and $\bar{g}\in
\bH^{\gamma+4}_{2,\theta}(T,\ell_2)$. Consequently, $v\in
\frH^{\gamma+5}_{2,\theta}(T)$ and $u\equiv0$.

4. The case $\gamma\in [-n-1,-n)$ with $n\in \{1,2,\cdots\}$ is
treated similarly. The theorem is proved.
\end{proof}

Now, we prove our results for the problem (\ref{eqn main system}) with
variable coefficients. For
  $n \in\bZ$, $\mu \in(0,1]$
 and $k=0,1,2,...$, we define
\begin{equation}
                           \label{eqn 5.6.2}
[u]^{(n)}_{k} =\sup_{\substack{x\in \bR^d_+\\
|\beta|=k}}(x^1)^{k+n}|D^{\beta}u(x)|,
\end{equation}
\begin{equation}
                              \label{eqn 5.6.3}
[u]^{(n)}_{k+\mu} =\sup_{\substack{x,y\in \bR^d_+
\\ |\beta|=k}}
(x^1 \wedge y^1)^{k+\mu+n}\frac{|D^{\beta}u(x)-D^{\beta}u(y)|}
{|x-y|^{\mu}},
\end{equation}
$$
|u|^{(n)}_{k}=\sum_{j=0}^{k}[u]^{(n)}_{j}, \quad |u|^{(n)}_{k+\mu}=
|u|^{(n)}_{k}+ [u]^{(n)}_{k+\mu}.
$$

Here is the main result of this section.

\begin{theorem}
                 \label{theorem half}
Let Assumptions \ref{main assumptions} and  \ref{assump 9.17} hold,
and
\begin{equation}
                       \label{eqn 9.11.1}
|a^{ij}_{kr}(t,\cdot)|^{(0)}_{|\gamma|_+}
+|b^{i}_{kr}(t,\cdot)|^{(1)}_{|\gamma|_+}
+|c_{kr}(t,\cdot)|^{(2)}_{|\gamma|_+} +
|\sigma^{i}_{kr}(t,\cdot)|^{(0)}_{|\gamma+1|_+} +
|\nu_{kr}(t,\cdot)|^{(1)}_{|\gamma+1|_+} \leq L
\end{equation}
and
\begin{equation}
                 \label{eqn 9.16.1}
 |a^{ij}_{kr}(t,x)-a^{ij}_{kr}(t,y)|
+|\sigma^{i}_{kr}(t,x)-\sigma^{i}_{kr}(t,y)|_{\ell_2} +
|Mb^i_{kr}(t,x)|+|M^2c_{kr}(t,x)|+|M\nu_{kr}(t,x)|_{\ell_2}<\kappa
\end{equation}
for all $x,y\in \bR^d_+$ with $|x-y|\leq x^1\wedge y^1$. Then there
exists $\kappa_0=\kappa_0(d,d_1,\theta,\delta, \gamma, K,L)$ so that if
$\kappa\leq \kappa_0$, then for any $f\in
M^{-1}\bH^{\gamma}_{2,\theta}(T)$, $g\in
\bH^{\gamma+1}_{2,\theta}(T,\ell_2)$ and $u_0\in
U^{\gamma+2}_{2,\theta}$, the problem (\ref{eqn main system})
defined on $\Omega\times [0,T]\times \bR^d_+$ admits a unique
solution $u\in \frH^{\gamma+2}_{2,\theta}(T)$, and it holds that
\begin{equation}
                                  \label{eqn 9.9.2}
 \|M^{-1}u\|_{\bH^{\gamma+2}_{2,\theta}(T)}\leq
 c\left(\|Mf\|_{\bH^{\gamma}_{2,\theta}(T)}
 +\|g\|_{\bH^{\gamma+1}_{2,\theta}(T,\ell_2)}+
 \|u_0\|_{U^{\gamma+2}_{2,\theta}}\right)
 \end{equation}
 where $c=c(d,d_1,\delta,K^j,L)$.
 \end{theorem}

\begin{remark}
See Remark \ref{05.18.01}(i) for the better understanding of the  condition (\ref{eqn 9.16.1}).
\end{remark}

\begin{remark}
Since $C^{\infty}_0$ is dense in $H^{\gamma}_{p,\theta}$, zero
boundary condition is implicitly imposed in Theorem \ref{theorem
half} (and in Theorem \ref{main theorem on domain} below).
\end{remark}

 To prove Theorem \ref{theorem half} we use
the following  three lemmas  taken from \cite{KK2}.

\begin{lemma}
                                         \label{lemma 8.26.10}
Let constants $C,\delta$ be in $(0,\infty)$, and $q$ be the smallest
integer such that $|\gamma|+2\leq q$.

(i) Let $\eta_{n}\in C^{\infty}(\bR^{d}_{+})$, $n=1,2,...$, satisfy
\begin{equation}
                                                   \label{8.26.11}
\sum_{n}M^{|\alpha|} |D^{\alpha} \eta_n |\leq C
\quad\text{in}\quad\bR^{d}_{+}
\end{equation}
for    any multi-index $\alpha$ such that $ 0\leq |\alpha| \leq q$.
Then for any $u\in H^{\gamma}_{p,\theta} $
$$
\sum_{n} \|\eta_{n}u\|^{p}_{H^{\gamma}_{p,\theta} } \leq
NC^{p}\|u\|^{p}_{H^{\gamma}_{p,\theta} },
$$
where the constant $N$ is independent of $u$, $\theta$, and $C$.

(ii) If, in addition to the condition in (i), $\sum_{n} \eta_{n}
^{2}\geq\delta$ on $\bR^{d}_{+}$, then for any $u\in
H^{\gamma}_{p,\theta} $,
\begin{equation}
                                                   \label{11.25.1}
\|u\|^{p}_{H^{\gamma}_{p,\theta} }\leq N\sum_{n}
\|\eta_{n}u\|^{p}_{H^{\gamma}_{p,\theta} },
\end{equation}
where the constant $N$ is independent of $u$ and $\theta$.
\end{lemma}

The reason that the first inequality in (\ref{11.14.1}) below is
written for $\eta_n^4$ (not for $\eta_n^2$ as in the above lemma) is
to have the possibility to apply Lemma \ref{lemma 8.26.10} to
$\eta_n^2$. Also, note $\sum a^2 \leq (\sum |a|)^2$.

\begin{lemma}
                                            \label{lemma 11.14.1}
For each $\varepsilon>0$ and $q=1,2,...$, there exist non-negative
functions $\eta_{n}\in C^{\infty}_{0}(\bR^{d}_{+})$, $n=1,2,...$
such that (i) on $\bR^{d}_{+}$ for each multi-index $\alpha$ with
$1\leq|\alpha|\leq q$ we have
\begin{equation}
                                             \label{11.14.1}
\sum_{n}\eta^{4}_{n}\geq1,\quad \sum_{n} \eta _{n} \leq N(d),
\quad\sum_{n}M^{|\alpha|}|D^{\alpha}\eta_{n}|\leq\varepsilon;
\end{equation}

(ii) for any $n$ and $x,y\in\text{\rm supp}\,\eta_{n}$ we have $
|x-y|\leq N ( x^{1}\wedge y^{1})$, where
$N=N(d,q,\varepsilon)\in[1,\infty)$.

\end{lemma}

\begin{lemma}
                                           \label{lemma 1.2.2}
Let $p\in(1,\infty)$, $\gamma,\theta\in\bR$. Then there exists a
constant $N=N(\gamma,|\gamma|_+,p,d)$   such that  if $f\in
H^{\gamma}_{p,\theta}$ and $a$ is a function with the finite norm
$|a|^{(0)}_{|\gamma|_+}$, then
 \begin{equation}
                                                 \label{8.19.05}
\|af\|_{H^{\gamma}_{p,\theta}} \leq N |a|^{(0)}_{|\gamma|_+}
\|f\|_{H^{\gamma }_{p,\theta}}.
\end{equation}
In addition,

(i) if $\gamma=0,1,2,...$, then
 \begin{equation}
                                            \label{1.24.06}
\|af\|_{H^{\gamma}_{p,\theta}} \leq N_1 \sup_{\bR^{d}_{+}}|a|\,
\|f\|_{H^{\gamma }_{p,\theta}}+ N_2\|f\|_{H^{\gamma-1}_{p,\theta}}
\sup_{\bR^{d}_{+}}\sup_{1\leq|\alpha|\leq\gamma}
|M^{|\alpha|}D^{\alpha}a |,
\end{equation}
where, obviously, one can take $N_1=1$ and  $N_2=0$ if $\gamma=0$.

(ii) if $\gamma$ is not integer, then
\begin{equation}
                                            \label{1.24.07}
\|af\|_{H^{\gamma}_{p,\theta}} \leq N  (\sup_{\bR^{d}_{+}}|a|)^{s}
(|a|^{(0)}_{|\gamma|_+})^{1-s} \|f\|_{H^{\gamma }_{p,\theta}},
\end{equation}
where $s:=1-\frac{|\gamma|}{{|\gamma|_+}} > 0$.

The same assertions hold true for $\ell_2$-valued $a$.
\end{lemma}

{\bf{Proof of Theorem \ref{theorem half}}} \;\;We proceed as in
Theorem 2.16 of \cite{KK}, where the theorem is proved for single
equations. As usual, for simplicity, we assume $u_0 =0$ (see the proof of Theorem 5.1 in \cite{Kr99}).
Also  having  the method of continuity in mind, we convince
ourselves  that
 to prove the theorem  it  suffices to
show that there exists  $\kappa_{0}$  such that
 the a priori estimate (\ref{eqn 9.9.2}) holds
 given that the solution already exists and $\kappa
\leq\kappa_{0}$.
  We divide the
proof into $6$ cases. The reason for this is that if $\gamma$ is
not an integer we  use (\ref{1.24.07}) and if $\gamma$ is a
non-negative integer we use (\ref{1.24.06}),
  but if $\gamma$ is a negative integer we  use the
 somewhat different approaches used in
 \cite{KK}.

 {\bf Case 1}: $|\gamma|\not\in\{0,1,2,...\}$.
 Take the least integer $q\geq|\gamma|+4$. Also take an
$\varepsilon\in(0,1)$ which will be specified later, and take a
sequence of functions $\eta_{n}$, $n=1,2,...$ from Lemma \ref{lemma
11.14.1} corresponding to $\varepsilon,q$.
 Then by Lemma \ref{lemma 8.26.10}, we have
\begin{equation}
                                                \label{8.28.15}
\|M^{-1}u\|_{\bH^{\gamma+2}_{2,\theta}(T)}^{2} \leq
N\sum_{n=1}^{\infty}
\|M^{-1}u\eta^{2}_{n}\|_{\bH^{\gamma+2}_{2,\theta}(T)}^{2}.
\end{equation}
 For any $n$ let  $x_{n}$ be a point in $\text{supp}\,\eta_{n}$
and $a^{ij}_{kr,n}(t)=a^{ij}_{kr}(t,x_{n})$, $
\sigma^i_{kr,n,m}(t)=\sigma^i_{kr,m}(t,x_{n})$.
 From  (\ref{eqn main system}), we easily have
$$
 d(u^k\eta^{2}_{n})=
(a^{ij}_{kr,n}(u^r\eta^{2}_{n})_{ x^ix^j}+M^{-1}f^k_{n})dt +
(\sigma^{i}_{kr,n,m}(u^r\eta^{2}_{n})_{x^i}+ g^k_{n,m})\, dw^m_{t},
$$
where
$$
f^k_{n}=(a^{ij}_{kr}-a^{ij}_{kr,n}) \eta^{2}_{n} Mu^r_{x^ix^j}
 -2a^{ij}_{kr,n}M(\eta^{2}_{n})_{ x^i}u^r_{x^j}
-a^{ij}_{kr,n}M^{-1}u^rM^{2}(\eta^{2}_{n})_{ x^ix^j}
$$
$$
+\eta_{n}^{2}Mb^{i}_{kr}u^r_{x^{i}}
+\eta_{n}^{2}M^{2}c_{kr}M^{-1}u^r +Mf^k\eta^{2}_{n},
$$
$$
g^k_{n,m}=(\sigma^{i}_{kr,m}-\sigma^{i}_{kr,n,m})\eta^{2}_{n}u^r_{x^i}-\sigma^{i}_{kr,n,m}M^{-1}u^r
M(\eta^2_{n})_{x^i} + M\nu_{kr,m} M^{-1}u^r\eta^2_{n} +
g^k_m\eta^2_{n} .
$$

By Theorem \ref{theorem half-constant}, for each $n$,
\begin{equation}
                                                 \label{8.28.20}
\|M^{-1}u\eta^{2}_{n}\|_{\bH^{\gamma+2}_{2,\theta}(T)}^{2} \leq N
(\|f_{n}\|^{2}_{\bH^{\gamma}_{2,\theta}(T)}+
\|g_{n}\|^{2}_{\bH^{\gamma+1}_{2,\theta}(T,\ell_2)})
\end{equation}
and  by (\ref{1.24.07}),
\begin{equation}
                                                             \label{1.24.01}
\|(a^{ij}_{kr}-a^{ij}_{kr,n}) \eta^{2}_{n} Mu_{x^ix^j}\|
_{\bH^{\gamma}_{2,\theta}(T)} \leq N\| \eta _{n} Mu_{xx}\|
_{\bH^{\gamma}_{2,\theta}(T)} \sup_{\omega,t,x}|(a^{ij}_{kr}-
a^{ij}_{kr,n})\eta_{n}|^{s},
\end{equation}
 where $s >0$ is a constant depending only on $\gamma$ and $|\gamma|_+$.

By Lemma \ref{lemma 11.14.1} (ii), for each   $n$ and
$x,y\in\text{supp}\,\eta_{n}$
  we have
$|x-y|\leq N(\varepsilon)(x^{1}\wedge y^{1})$, where
$N(\varepsilon)=N(d,q,\varepsilon)$, and we can easily fix points
$x_i$ lying on the straight segment
 connecting $x$ and $y$ and including $x$ and $y$ so that
the number of points are not more than $N(\varepsilon)+2 \leq
3N(\varepsilon)$ and $|x_i-x_{i+1}|\leq x^1_{i} \wedge x^1_{i+1}$.
It follows from our assumptions
$$
\sup_{\omega,t,x}|(a^{ij}_{kr}-a^{ij}_{kr,n})\eta_{n}|
 \leq 3N(\varepsilon)\kappa.
$$
We substitute this  to (\ref{1.24.01})  and get
$$
\|(a^{ij}_{kr}-a^{ij}_{krn}) \eta^{2}_{n}
Mu^r_{x^ix^j}\|_{\bH^{\gamma}_{2,\theta}(T)} \leq N
N(\varepsilon)\kappa^{s}\| \eta _{n}
Mu_{xx}\|_{\bH^{\gamma}_{2,\theta}(T)}.
$$
Similarly,
$$
\| \eta^{2}_{n} Mb^{i}_{kr}u^r_{x^i}\|
_{\bH^{\gamma}_{2,\theta}(T)}+ \| \eta^{2}_{n}
M^{2}c_{kr}M^{-1}u^r\| _{\bH^{\gamma}_{2,\theta}(T)}+
\|(\sigma^{i}_{kr}-\sigma^{i}_{kr,n})\eta^{2}_{n}u^r_{x^i}\|_{\bH^{\gamma+1}_{2,\theta}(T,\ell_2)}
$$
$$
+ \|\eta^{2}_{n} M\nu_{kr}
M^{-1}u^r\|_{\bH^{\gamma+1}_{2,\theta}(T,\ell_2)} \leq N
N(\varepsilon)
\kappa^{s}\left(\|\eta_{n}u_{x}\|_{\bH^{\gamma+1}_{2,\theta}(T)}
+\|\eta_{n}M^{-1}u\| _{\bH^{\gamma+1}_{2,\theta}(T)}\right).
$$
Coming back to (\ref{8.28.20}) and (\ref{8.28.15}) and using Lemma
\ref{lemma 8.26.10}, we conclude
$$
\|M^{-1}u\|_{\bH^{\gamma+2}_{2,\theta}(T)}^{2} \leq NN(\varepsilon)
\kappa^{2s}\left(\|M u_{xx}\|_{\bH^{\gamma}_{2,\theta}(T)}^{2} +
\|u_{x}\|_{\bH^{\gamma+1}_{2,\theta}(T)}^{2} +
\|M^{-1}u\|_{\bH^{\gamma+1}_{2,\theta}(T)}^{2}\right)
$$
\begin{equation}
                                               \label{11.19.2}
+NC^{2}\left(\|u_{x}\|_{\bH^{\gamma }_{2,\theta}(T)}^{2}
+\|M^{-1}u\|_{\bH^{\gamma+1}_{2,\theta} (T)}^{2}\right)+N \left(
\|Mf\|_{\bH^{\gamma }_{2,\theta}(T)}^{2}+
\|g\|^2_{\bH^{\gamma+1}_{2,\theta}(T,\ell_2)}\right),
\end{equation}
where
$$
C=\sup_{\bR^{d}_{+}}\sup_{|\alpha|\leq q-2}
\sum_{n=1}^{\infty}M^{|\alpha|}(|D^{\alpha}(M(\eta_{n}^{2})_{x})|
+|D^{\alpha}(M^{2}(\eta_{n}^{2})_{xx})|).
$$
By construction, we have $C\leq N\varepsilon$. Furthermore (see
Lemma \ref{lemma 2})
\begin{equation}
                                               \label{11.19.1}
\|u_{x}\|_{H^{\gamma+1}_{2,\theta}} \leq
N\|M^{-1}u\|_{H^{\gamma+2}_{2,\theta}},\quad
\|Mu_{xx}\|_{H^{\gamma}_{2,\theta}} \leq
N\|M^{-1}u\|_{H^{\gamma+2}_{2,\theta}}\;\;.
\end{equation}

Hence (\ref{11.19.2}) yields
$$
\|M^{-1}u\|_{\bH^{\gamma+2}_{2,\theta}(T)}^{2} \leq
N_{1}(N(\varepsilon)\kappa^{2s}+\varepsilon^{2})
\|M^{-1}u\|_{\bH^{\gamma+2}_{2,\theta}(T)}^{2} +N \left(
\|Mf\|_{\bH^{\gamma }_{2,\theta}(T)}^{2}+
 \|g\|^2_{\bH^{\gamma+1}_{2,\theta}(T,\ell_2)} \right).
$$
Finally,  to get a priori estimate (\ref{eqn 9.9.2}) it's enough to
choose first
 $\varepsilon$ and then $\kappa_{0}$
so that $N_{1}(N(\varepsilon)\kappa^{2s} +\varepsilon^{2})\leq1/2$
for $\kappa\leq\kappa_{0}$.

 {\bf Case 2}: $\gamma =0$.
 Proceed as in Case 1 with
$\varepsilon=1$ and arrive at (\ref{8.28.20}) which is
$$
\|M^{-1}u\eta^{2}_{n}\|_{\bH^{2}_{2,\theta}(T)}^{2} \leq
N\left(\|f_{n}\|^{2}_{\bL_{2,\theta}(T)}+
\|g_{n}\|^{2}_{\bH^{1}_{2,\theta}(T,\ell_2)}\right).
$$

Notice that (\ref{1.24.01}) holds with $s=1$ (since $\gamma=0$).
Also by
 (\ref{1.24.06}),
$$
\|(\sigma^{i}_{kr}-\sigma^{i}_{kr,n})\eta^2_{n}u^r_{x^i}\|_{\bH^{1}_{2,\theta}(T,\ell_2)}
\leq N \sup_{\omega,t,x}
|(\sigma^{i}_{kr}-\sigma^{i}_{kr,n})\eta_{n}|_{\ell_2}
\|\eta_{n}u_x\|_{\bH^{1}_{2,\theta}(T)} +
 N \|\eta_{n}u_x\|_{\bL_{2,\theta}(T)}
$$
\begin{equation}
                                                         \label{1.24.02}
\leq  N \kappa \|\eta_{n}u_x\|_{\bH^{1}_{2,\theta}(T)} + N
\|\eta_{n}u_x\|_{\bL_{2,\theta}(T)}.
\end{equation}

From this point, by following the arguments in Case 1, one gets
 $$
\|M^{-1}u\|_{\bH^2_{2,\theta}(T)} \leq N_{1} \kappa
\|M^{-1}u\|_{\bH^2_{2,\theta}(T)}
 + N \|M^{-1}u\|_{\bH^1_{2,\theta}(T)} +
N\|Mf\|_{\bL_{2,\theta}(T)}+ N\|g\|_{\bH^1_{2,\theta}(T)}\;\;.
$$
Thus, if $N_{1} \kappa_0 \leq 1/2 $ and $\kappa \leq \kappa_0$,
then we have
\begin{equation}
                                                            \label{1.24.03}
\|M^{-1}u\|_{\bH^2_{2,\theta}(T)} \leq  N
\|M^{-1}u\|_{\bH^1_{2,\theta}(T)} + N\|Mf\|_{\bL_{2,\theta}(T)}+
N\|g\|_{\bH^1_{2,\theta}(T,\ell_2)}\;.
\end{equation}

Next, if necessary, by reducing $\kappa_0$ (note that we are free to
do this)
 we will estimate the norm
$\|M^{-1}u\|_{\bH^1_{2,\theta}(T)}$. Take an $\varepsilon \in (0,1)$
which will be specified later and proceed as in Case 1 and write
(\ref{8.28.15}) and (\ref{8.28.20}) for $\gamma=-1$. The latter is
$$
\|M^{-1}u\eta^2_{n}\|^2_{\bH^1_{2,\theta}(T)} \leq N\left(
\|f_n\|^2_{\bH^{-1}_{2,\theta}(T)}+
\|g_n\|^2_{\bL_{2,\theta}(T,\ell_2)}\right)\;.
$$
Using the fact $\|f_n\|_{\bH^{-1}_{2,\theta}(T)} \leq
\|f_n\|_{\bL_{2,\theta}(T)}$ and the previous arguments, one obtains
$$
\|M^{-1}u\|^2_{\bH^{1}_{2,\theta}(T)} \leq NN^2(\varepsilon)
\kappa^2 \left(\|Mu_{xx}\|^2_{\bL_{2,\theta}(T)}+
\|u_x\|^2_{\bL_{2,\theta}(T)} +
\|M^{-1}u\|^2_{\bL_{2,\theta}(T)}\right)
$$
$$
+ N C^2 \left(\|u_x\|^2_{\bL_{2,\theta}(T)}+
\|M^{-1}u\|^2_{\bL_{2,\theta}(T)}\right) + N
\left(\|Mf\|^2_{\bL_{2,\theta}(T)}+
\|g\|^2_{\bL_{2,\theta}(T,\ell_2)}\right),
$$
where $C$ is introduced after (\ref{11.19.2}). By using
(\ref{11.19.1}) we get
$$
\|M^{-1}u\|^2_{\bH^{1}_{2,\theta}(T)} \leq N (
N^2(\varepsilon)\kappa^{2}+\varepsilon^2)\|M^{-1}u\|^2_{\bH^2_{2,\theta}(T)}
+ N \left(\|Mf\|^2_{\bL_{2,\theta}(t)}+
\|g\|^2_{\bL_{2,\theta}(T,\ell_2)}\right).
$$
Finally, by substituting this into (\ref{1.24.03}) and then choosing
$\varepsilon$ and then $\kappa_0$ properly, one gets the desired
estimate.

{\bf{Case 3}}. $\gamma \in\{1,2,...\}$. Take $\kappa_0$ from Case 2 and assume $\kappa\leq \kappa_0$. Proceed as in Case 2 with $\varepsilon=1$. By  (\ref{1.24.06}),
$$
\|(a^{ij}_{kr}-a^{ij}_{kr,n}) \eta^{2}_{n} Mu_{x^ix^j}\|
_{\bH^{\gamma}_{2,\theta}(T)} \leq N \kappa \| \eta _{n} Mu_{xx}\|_{\bH^{\gamma}_{2,\theta}(T)}+ N\| \eta _{n} Mu_{xx}\|_{\bH^{\gamma-1}_{2,\theta}(T)} ,
$$
Similarly,
\begin{eqnarray*}
\|f_n\|_{\bH^{\gamma}_{2,\theta}(T)}+\|g_n\|_{\bH^{\gamma+1}_{2,\theta}(T,\ell_2)}&\leq& N \kappa \left(\|\eta_n Mu_{xx}\|_{\bH^{\gamma}_{2,\theta}(T)}+\|\eta_{n}u_{x}\|_{\bH^{\gamma+1}_{2,\theta}(T)}
+\|\eta_{n}M^{-1}u\| _{\bH^{\gamma+1}_{2,\theta}(T)}\right)
\\
&+& N \left(\|\eta_n Mu_{xx}\|_{\bH^{\gamma-1}_{2,\theta}(T)}+\|\eta_{n}u_{x}\|_{\bH^{\gamma}_{2,\theta}(T)}
+\|\eta_{n}M^{-1}u\| _{\bH^{\gamma}_{2,\theta}(T)}\right).
\end{eqnarray*}
This easily leads to
$$
\|M^{-1}u\|_{\bH^{\gamma+2}_{2,\theta}(T)}\leq N_2\kappa \|M^{-1}u\|_{\bH^{\gamma+2}_{2,\theta}(T)}+N\|M^{-1}u\|_{\bH^{\gamma+1}_{2,\theta}(T)} + N \|Mf\|_{\bH^{\gamma}_{2,\theta}(T)}+N\|g\|_{\bH^{\gamma+1}_{2,\theta}(T,\ell_2)}.
$$
Now additionally  assume $N_2\kappa\leq 1/4$. Then it is enough  to use the  interpolation inequality (\cite{kr99-1}, Theorem 2.6)
 $$\|M^{-1}u\|_{H^{\gamma+1}_{2,\theta}}\leq \varepsilon \|M^{-1}u\|_{H^{\gamma+2}_{2,\theta}}+N(\varepsilon,\gamma)\|M^{-1}u\|_{H^2_{2,\theta}}$$
  and the results in Case 2.

{{\bf Case 4}: $\gamma=-1$.   We temporarily assume that (\ref{eqn
9.11.1})
 holds with $\gamma=1$.  In this case
we prove the theorem directly without depending on an a priori
estimate. Take $\kappa_0$  which corresponds to the case $\gamma=0$.
Assume $\kappa\leq \kappa_0$, then  the operator $\cR$ which maps
the couples $(f,g) \in M^{-1}\bL_{2,\theta}(T) \times
\bH^{1}_{2,\theta}(T,\ell_2)$ into the solutions $u \in
\frH^2_{2,\theta}(T)$ of the problem (\ref{eqn main system}) defined
on $\Omega\times [0,T]\times \bR^d_+$ with zero initial data is
well-defined and bounded.

Now take $(f,g)\in M^{-1}\bH^{-1}_{2,\theta}(T)\times
\bL_{2,\theta}(T,\ell_2)$. By Corollary
 2.12 in \cite{kr99}  we have the following representations
\begin{equation}
                                                            \label{1.25.04}
f=MD_{\ell}f^{\ell}, \quad g= MD_{\ell}g^{\ell},
\end{equation}
where $f^{\ell}=(f^{\ell,1},\cdots,f^{\ell,d_1}) \in
M^{-1}\bL_{2,\theta}(T),
g^{\ell}=(g^{\ell,1},\cdots,g^{\ell,d_1})\in
\bH^1_{2,\theta}(T,\ell_2), \ell=1,2,...,d$ and
\begin{equation}
                                                                 \label{1.25.05}
\sum_{\ell=1}^{d}\|Mf^{\ell}\|_{\bL_{2,\theta}(T)} \leq N
\|Mf\|_{\bH^{-1}_{2,\theta}(T)}, \quad
\sum_{\ell=1}^{d}\|g^{\ell}\|_{\bH^{1}_{2,\theta}(T,\ell_2)}
  \leq N \|g\|_{\bL_{2,\theta}(T,\ell_2)}.
\end{equation}

Next denote $v^{\ell}=(v^{\ell,1},\cdots v^{\ell,d_1})=\cR
(f^{\ell},g^{\ell})$ and
$\bar{v}=(\bar{v}^1,\cdots,\bar{v}^{d_1})=\sum_{\ell=1}^d
MD_{\ell}v^{\ell}$. Then by (\ref{11.19.1}), $\bar{v}$ is in $
M\bH^{1}_{2,\theta}(T)$ and satisfies
$$
d\bar{v}^k=(a^{ij}_{kr}\bar{v}^r_{x^i x^j}+
b^{i}_{kr}\bar{v}^r_{x^i}+c_{kr}\bar{v}^r+f^k+\bar{f}^k)\, dt +
(\sigma^{i}_{kr,m}\bar{v}^r_{x^i}+ \nu_{kr,m}\bar{v}^r+g^{k}_m+
\bar{g}^k_m)\, dw^m_t,
$$
where
$$
\bar{f}^k=(MD_{\ell}a^{ij}_{kr})v^{\ell,r}_{x^i
x^j}-2a^{i1}_{kr}v^{\ell,r}_{x^{\ell}x^i}+ (M^2
D_{\ell}b^{i}_{kr})M^{-1}v^{\ell,r}_{x^i}- Mb^{1}_{kr}
M^{-1}v^{\ell,r}_{x^{\ell}} + (M^{3}D_{\ell}c_{kr})M^{-2}v^{\ell},
$$
$$
\bar{g}^k = (M D_{\ell}\sigma^{i}_{kr}) v^{\ell,r}_{x^i} -
\sigma^{1}_{kr}v^{\ell,r}_{x^{\ell}} + (M^2 D_{\ell} \nu_{kr})
M^{-1}v^{\ell,r}.
$$

By assumptions one can easily check  that  $|\cdot|^{(0)}_{0}$-norm
of $MD_{\ell}a^{ij}_{kr}$, $M^2D_{\ell}b^{i}_{kr}$,
$M^3D_{\ell}c_{kr}$ and $|\cdot|^{(0)}_{1}$- norm of
$MD_{\ell}\sigma^i_{kr}$, $M^2D_{\ell}\nu_{kr}$ are finite.
Therefore
$$
M\bar{f}\in \bL_{2,\theta}(T),\quad \bar{g}\in
\bH^{1}_{2,\theta}(T,\ell_2).
$$
Finally we define $\bar{u}= \cR(\bar{f},\bar{g})$ and $u:= \bar{v}-
\bar{u}$. Then $u \in \frH^{1}_{2,\theta}(T)$ satisfies (\ref{eqn
main system}) and the a priori estimate follows from the formulas
defining $\bar{u}$ and $\bar{v}$.

Next, we  prove the uniqueness of solutions. Let $\kappa\leq
\kappa_0$ with $\kappa_0$ found above for the case $\gamma=0$ and
assume
 $u \in \frH^{1}_{2,\theta}(T)$ satisfies (\ref{eqn main system})
 with $f=0,g^k=0$ and $u_0=0$. Since we already have
the uniqueness in the space $\frH^{2}_{2,\theta}(T)$, to show
$u\equiv 0$ we only need to show $u\in \frH^2_{2,\theta}(T)$. Take
$\eta_{n}$ from Lemma \ref{lemma 11.14.1} corresponding to
$\varepsilon =1$. From (\ref{eqn main system}) one can write the
system for $\eta_n u$ for each $n$ and get
$$
d(\eta_{n}u^k)=\left( a^{ij}_{kr}(\eta_{n}u^r)_{x^i x^j}+
b^{i}_{kr}(\eta_{n}u^r)_{x^i}+ c_{kr}(\eta_n u^r) + \tilde{f}^k_n
\right) \,dt
$$
$$
 + \left(\sigma^{i}_{kr,m}(\eta_{n}u^r)_{x^i}+ \nu_{kr,m}(\eta_{n}u^r)
  + \tilde{g}^k_{n,m} \right)\,dw^m_t,
$$
where
$$
\tilde{f}^k_n=-2a^{ij}_{kr}\eta_{n
x^i}u^r_{x^j}-(a^{ij}_{kr}\eta_{nx^i x^j}+
b^{i}_{kr}\eta_{nx^i})u^r, \quad \tilde{g}^k_{n,m}=
-\sigma^{i}_{kr,m} (\eta_n)_{x^i}u^r.
$$
Since $u\in M\bH^{1}_{2,\theta}(T)$ and $\eta_n$ has compact
support, we easily have $(\tilde{f},\tilde{g}) \in \bL_{2}(T)\times
\bH^1_2(T,\ell_2)$.  Also the above system will not change if we
arbitrarily change
$a^{ij}_{kr},b^{i}_{kr},c_{kr},\sigma^i_{kr},\nu_{kr}$ outside of
the support of $\eta_n$. Therefore using Theorem \ref{thm 2}, one
easily
 concludes that
  $\eta_n u \in \bH^2_2(T)$ and hence
  $M^{-1}\eta_{n}u\in \bH^2_{2,\theta}(T),
\eta_n u \in \frH^2_{2,\theta}(T)$.
  Then finally by using (\ref{eqn 9.9.2}) (which we have for $\gamma=0$)
  and Lemma \ref{lemma 8.26.10} one obtains
  $\|M^{-1}u\|_{\bH^2_{2,\theta}(T)}< \infty$,
that is, $u \in \frH^2_{2,\theta}(T)$.

{{\bf Case 5}: $\gamma=-1$ with no additional assumptions}. To prove
the a priori estimate we use the results of Case 3.  Fix a
non-negative smooth function $\phi \in C^{\infty}_{0}(B_{1/2}(0))$
with a unit integral. Define
$$
\bar{\sigma}(x)= \int \sigma(y)(x^{1})^{-d}\phi(\frac{x-y}{x^1})\,dy,
$$
and define $\bar{\nu}$ similarly. Observe that
$$
|\bar{\sigma}-\sigma| \leq \kappa, \quad |M\bar{\nu}| \leq 2 \kappa.
$$

Also using the fact $x^1 \leq 2(x^1-x^1z^1) \leq 4 x^1$ for
$|z^1|\leq 1/2$,
 one can easily check
that there is a constant $N_0< \infty$ such that
$$
|\bar{\sigma}|^{(0)}_{2} + |\bar{\nu}|^{(1)}_{2} < N_0.
$$
For instance, let  $i,j \geq 2$, and  $\delta_{1\ell}=1$ if $\ell=1$
and $\delta_{1\ell}=0$ otherwise,  then
$$
x^1 \bar{\sigma}_{x^1}(x)= \int_{|z|\leq 1/2}\sigma(x-x^1z)
[-d\phi(z)+\phi_{x^{\ell}}(z)\cdot (\delta_{1\ell}-z^{\ell})]\,dz,
$$
$$
(x^1)^2 \bar{\nu}_{x^1}(x)=\int_{|z|\leq 1/2} x^1 \nu (x-x^1z)
[-d\phi(z)+\phi_{x^{\ell}}(z)\cdot (\delta_{1\ell}-z^{\ell})]\,dz,
$$
$$
(x^1)^2 \bar{\sigma}_{x^i x^j}(x)=
 \int_{|z|\leq 1/2} \sigma(x-x^1z)\phi_{x^i x^j}(z)dz,
$$
$$
  (x^1)^3 \bar{\nu}_{x^i x^j}(x)=
 \int_{|z|\leq 1/2} x^1 \nu (x-x^1z)\phi_{x^i x^j}(z)dz,
$$
and therefore it is obvious that the  functions  above  are bounded.
Also, all other cases  can be considered similarly.

Take $(f,g) \in M^{-1}\bH^{-1}_{2,\theta}(T)\times
\bL_{2,\theta}(T,\ell_2)$ and let $u \in \frH^{1}_{2,\theta}(T)$ be
a solution of (\ref{eqn main system}) with zero initial data. Then
$$
du^k=(a^{ij}_{kr}u^r_{x^i x^j}+ b^i_{kr} u^r_{x^i}+c_{kr}u^r
+f^k)\,dt +( \bar{\sigma}^{i}_{kr,m}u^r_{x^i}+\bar{\nu}_{kr,m}u^r +
\bar{g}^k_m) \, dw^m_t,
$$
where  $\bar{g}^k=g^k+(\sigma^i_{kr}- \bar{\sigma}^i_{kr})u^r_{x^i}+
(\nu_{kr}-\bar{\nu}_{kr})u^r$. Note
\begin{equation}
                                                            \label{03.29.11}
\|\bar{g}\|_{\bL_{2,\theta}(T,\ell_2)}\leq
\|g\|_{\bL_{2,\theta}(T,\ell_2)}+ \kappa\|u_x\|_{\bL_{2,\theta}(T)}+
3\kappa\|M^{-1}u\|_{\bL_{2,\theta}(T)}.
\end{equation}
 Thus, by the results of Case 4,   if $\kappa \leq
\kappa_0$, then
\begin{eqnarray}
\|M^{-1}u\|_{\bH^{1}_{2,\theta}(T)} &\leq& N
 \left(\|Mf\|_{\bH^{-1}_{2,\theta}(T)}+\|\bar{g}\|_{\bL_{2,\theta}(T,\ell_2)}\right)\nonumber\\
&\leq&
N_{1}\left(\|Mf\|_{\bH^{-1}_{2,\theta}(T)}+\|g\|_{\bL_{2,\theta}(T,\ell_2)}
+\kappa \|M^{-1}u\|_{\bH^1_{2,\theta}(T)}\right), \label{1.27.02}
\end{eqnarray}
where the second inequality comes from  (\ref{03.29.11}) and  (\ref{11.19.1}). Finally we
assume
$$
\kappa \leq \kappa_0  \wedge (2 N_{1})^{-1}.
$$
Then   (\ref{1.27.02}) yields
$$
\|M^{-1}u\|_{\bH^1_{2,\theta}(T)}\leq
2N_{1}\left(\|Mf\|_{\bH^{-1}_{2,\theta}(T)}+\|g\|_{\bL_{2,\theta}(T,\ell_2)}\right).
$$
Thus we get the desired result for $\gamma=-1$.

{{\bf Case 6}: $\gamma=-2,-3,-4,...$}.  In this case it is enough to
repeat the processes in Case 4, but since $|\gamma|\geq |\gamma+2|$,
additional smoothness assumption on the coefficients is unnecessary.
The theorem is proved. $\Box$

\mysection{The system with bounded $C^1$-domain $\cO$}\label{main
section}

\begin{assumption}
                                         \label{assumption domain}

The bounded domain $\cO$  is of class $C^{1}_{u}$. In other words,
for any $x_0 \in \partial \cO$, there exist constants $r_0,
K_0\in(0,\infty) $ and  a one-to-one continuously differentiable
mapping $\Psi$ of
 $B_{r_0}(x_0)$ onto a domain $J\subset\bR^d$ such that

(i) $J_+:=\Psi(B_{r_0}(x_0) \cap \cO) \subset \bR^d_+$ and
$\Psi(x_0)=0$;

(ii)  $\Psi(B_{r_0}(x_0) \cap \partial \cO)= J \cap \{y\in
\bR^d:y^1=0 \}$;

(iii) $\|\Psi\|_{C^{1}(B_{r_0}(x_0))}  \leq K_0 $ and
$|\Psi^{-1}(y_1)-\Psi^{-1}(y_2)| \leq K_0 |y_1 -y_2|$ for any $y_i
\in J$;

(iv)   $\Psi_{x}$ is uniformly continuous in for $B_{r_{0}}(x_{0})$.
\end{assumption}

To proceed further we introduce some well known results from
\cite{GH} and \cite{KK2} (also, see \cite{La} for the details).

\begin{lemma}
                                           \label{lemma 10.3.1}
Let the domain $\cO$ be of class $C^{1}_{u}$. Then

(i) there is a bounded real-valued function $\psi$ defined in
$\bar{\cO} $, the closure of $\cO$,  such that the functions $\psi(x)$ and
$\rho(x):=\text{dist}(x,\partial \cO)$ are comparable. In other
words, $N^{-1}\rho(x) \leq \psi(x) \leq N\rho(x)$ with some constant
 $N$ independent of $x$,

 (ii) for any  multi-index $\alpha$,
\begin{equation}
                                                             \label{03.04.01}
\sup_{\cO} \psi ^{|\alpha|}(x)|D^{\alpha}\psi_{x}(x)| <\infty.
\end{equation}

\end{lemma}

Now, we take the Banach spaces introduced in \cite{KK2} and
\cite{Lo2}.
 Let $\zeta\in C^{\infty}_{0}(\bR_{+})$
be a   function satisfying (\ref{eqn 5.6.5}). For $x\in \cO$ and
$n\in\bZ=\{0,\pm1,...\}$ we define
$$
\zeta_{n}(x)=\zeta(e^{n}\psi(x)).
$$
Then  we have $\sum_{n}\zeta_{n}\geq c$ in $\cO$ and
\begin{equation*}
\zeta_n \in C^{\infty}_0(\cO), \quad |D^m \zeta_n(x)|\leq
N(m)e^{mn}.
\end{equation*}
For $\theta,\gamma \in \bR$, let $H^{\gamma}_{p,\theta}(\cO)$ be the
set of all distributions $u=(u^1,u^2,\cdots u^{d_1})$  on $\cO$ such
that
\begin{equation}
                                                 \label{10.10.03}
\|u\|_{H^{\gamma}_{p,\theta}(\cO)}^{p}:= \sum_{n\in\bZ} e^{n\theta}
\|\zeta_{-n}(e^{n} \cdot)u(e^{n} \cdot)\|^p_{H^{\gamma}_p} < \infty.
\end{equation}
 If $g=(g^1,g^2,\ldots,g^{d_1})$ and each $g^k$ is an
 $\ell_2$-valued function,
 then we define
$$
\|g\|_{H^{\gamma}_{p,\theta}(\cO,\ell_2)}^{p}= \sum_{n\in\bZ}
e^{n\theta} \|\zeta_{-n}(e^{n} \cdot)g(e^{n}
\cdot)\|^p_{H^{\gamma}_p(\ell_2)}.
$$
It is known (see, for instance, \cite{Lo2}) that up to equivalent
norms the space $H^{\gamma}_{p,\theta}(\cO)$ is independent of the
choice of $\zeta$ and $\psi$. Moreover if $\gamma$ is a non-negative
integer, then
$$
H^{\gamma}_{p,\theta}(\cO)=\{u:u,\psi
Du,\cdots,\psi^{|\alpha|}D^{\alpha}u\in
L_p(\:\cO,\psi^{\theta-d}dx\:),\;|\alpha|\le\gamma\},
$$
$$
\|u\|^p_{H^{\gamma}_{p,\theta}}(\cO)\sim \sum_{|\alpha|\le\gamma}
\int_{\cO}|\psi^{|\alpha|} D^{\alpha}u(x)|^p \psi^{\theta-d} \,dx.
$$

Denote $\rho(x,y)=\rho(x)\wedge \rho(y)$ and
$\psi(x,y)=\psi(x)\wedge \psi(y)$. For
  $n \in\bZ$, $\mu \in(0,1]$
 and $k=0,1,2,...$, we define
$$
|u|_{C}=\sup_{\cO}|u(x)|, \quad [u]_{C^{\mu}}=\sup_{x\neq
y}\frac{|u(x)-u(y)|}{|x-y|^{\mu}}.
$$
\begin{equation}
                           \label{eqn 5.6.2}
[u]^{(n)}_{k}=[u]^{(n)}_{k,\cO} =\sup_{\substack{x\in \cO\\
|\beta|=k}}\psi^{k+n}(x)|D^{\beta}u(x)|,
\end{equation}
\begin{equation}
                              \label{eqn 5.6.3}
[u]^{(n)}_{k+\mu}=[u]^{(n)}_{k+\mu,\cO} =\sup_{\substack{x,y\in \cO
\\ |\beta|=k}}
\psi^{k+\mu+n}(x,y)\frac{|D^{\beta}u(x)-D^{\beta}u(y)|}
{|x-y|^{\mu}},
\end{equation}
$$
|u|^{(n)}_{k}=|u|^{(n)}_{k,\cO}=\sum_{j=0}^{k}[u]^{(n)}_{j,\cO},
\quad |u|^{(n)}_{k+\mu}=
 |u|^{(n)}_{k+\mu,\cO}=|u|^{(n)}_{k, \cO}+
[u]^{(n)}_{k+\mu,\cO}\;.
$$
Remember that in case $\cO=\bR^d_+$, to define
$|u|^{(n)}_{k}=|u|^{(n)}_{k,\bR_+}$, we used   $\rho(x)(=x^1)$ and
$\rho(x)\wedge \rho(y)$ in place of $\psi(x)$ and $\psi(x,y)$
respectively in (\ref{eqn 5.6.2}) and (\ref{eqn 5.6.3}).

Below we collect some other properties of the  spaces
$H^{\gamma}_{p,\theta}(\cO)$.

\begin{lemma}(\cite{Lo2})
              \label{lemma 1}
(i) The assertions (i)-(iii) in Lemma \ref{lemma 2} hold if one
formally replace
     $M$ and $H^{\gamma}_{p,\theta}$ by $\psi$ and
     $H^{\gamma}_{p,\theta}(\cO)$, respectively.

(ii) There is a constant $N=N(\gamma,|\gamma|_+,p,\theta)>0$ so that
$$
\|af\|_{H^{\gamma}_{p,\theta}(\cO)}\leq N
|a|^{(0)}_{|\gamma|_+}\|f\|_{H^{\gamma}_{p,\theta}(\cO)}.
$$
\end{lemma}

Denote
$$
\bH^{\gamma}_{p,\theta}(\cO,T)=L_p(\Omega\times
[0,T],\cP,H^{\gamma}_{p,\theta}(\cO)), \quad
\bH^{\gamma}_{p,\theta}(\cO,T,\ell_2)=L_p(\Omega\times
[0,T],\cP,H^{\gamma}_{p,\theta}(\cO,\ell_2)),
$$
$$
U^{\gamma}_{p,\theta}(\cO)=
\psi^{1-2/p}L_{p}(\Omega,\cF_0,H^{\gamma-2/p}_{p,\theta}(\cO)),
\quad \bL_{p,\theta}(\cO,T)=\bH^{0}_{p,\theta}(\cO,T).
$$
\begin{definition}
We say $ u\in \frH^{\gamma+2}_{p,\theta}(\cO,T)$ if $u=(u^1,\cdots,
u^{d_1})\in \psi\bH^{\gamma+2}_{p,\theta}(\cO,T)$, $u(0,\cdot) \in
U^{\gamma+2}_{p,\theta}(\cO)$ and  for some $f \in
\psi^{-1}\bH^{\gamma}_{p,\theta}(\cO,T)$, $g\in
\bH^{\gamma+1}_{p,\theta}(\cO,T,\ell_2)$,
$$
du= f \,dt + g_m \, dw^m_t,
$$
in the sense of distributions. The norm in  $
\frH^{\gamma+2}_{p,\theta}(\cO,T)$ is defined by
$$
\|u\|_{\frH^{\gamma+2}_{p,\theta}(\cO,T)}=
\|\psi^{-1}u\|_{\bH^{\gamma+2}_{p,\theta}(\cO,T)} + \|\psi
f\|_{\bH^{\gamma}_{p,\theta}(\cO,T)}  +
\|g\|_{\bH^{\gamma+1}_{p,\theta}(\cO,T)} +
\|u(0,\cdot)\|_{U^{\gamma+2}_{p,\theta}(\cO)}.
$$

\end{definition}

The following result  is due to N.V.\:Krylov (see \cite{Kr01} and
\cite{Kim04-1}).

\begin{lemma}
                             \label{lemma 15.05}
 Let $p\geq 2$. The space $\frH^{\gamma+2}_{p,\theta}(T)$ is a Banach space and  there exists a constant
$c=c(d,p,\theta,\gamma,T)$ such that
$$
\bE \sup_{t\leq T}\|u(t)\|^p_{H^{\gamma+1}_{p,\theta}(\cO)}\leq c
\|u\|^p_{\frH^{\gamma+2}_{p,\theta}(\cO,T)}.
$$
In particular, for any $t\leq T$,
$$
\|u\|^p_{\bH^{\gamma+1}_{p,\theta}(\cO,t)}\leq c \int^t_0
\|u\|^p_{\frH^{\gamma+2}_{p,\theta}(\cO,s)}ds.
$$
\end{lemma}

\begin{assumption}
             \label{assumption regularity}
(i) The functions $a^{ij}_{kr}(t,\cdot),\sigma^{i}_{kr}(t,\cdot)$
 are {\bf{point-wise continuous}} in $\cO$.
  That is, for any $\varepsilon >0$ and
$x\in \cO$, there exists $\delta=\delta(\varepsilon,x)$ such that
$$
|a^{ij}_{kr}(t,x)-a^{ij}_{kr}(t,y)|
+|\sigma^{i}_{kr}(t,x)-\sigma^i_{kr}(t,y)|_{\ell_2}<\varepsilon
$$
whenever $x,y\in \cO$ and $|x-y|<\delta$.

(ii) There is  a control on the behavior of $a^{ij}_{kr}$,
$b^i_{kr}$, $c_{kr}$, $\sigma^{i}_{kr}$
 and $\nu_{kr}$ near
$\partial \cO$, namely,
\begin{equation}
                                                \label{12.10.1}
\lim_{\substack{\rho(x)\to 0\\
x\in \cO}}\sup_{\substack{y\in \cO \\|x-y|\leq\rho(x,y)}} \sup_{t,
\omega} [|a^{ij}_{kr}(t,x)-a^{ij}_{kr}(t,y) |+
|\sigma^{i}_{kr}(t,x)-\sigma^{i}_{kr}(t,x)|_{\ell_2}] =0.
\end{equation}
\begin{equation}
                                                       \label{05.04.01}
\lim_{\substack{\rho(x)\to0\\
x\in \cO}} \sup_{t,
\omega}[\rho(x)|b^i_{kr}(t,x)|+\rho^{2}(x)|c_{kr}(t,x)|+
\rho(x)|\nu_{kr}(t,x)|_{\ell_2}]=0.
\end{equation}

(iii) For any $t>0$ and $\omega\in \Omega$,
$$
|a^{ij}_{kr}(t,\cdot)|^{(0)}_{|\gamma|_+}
+|b^{i}_{kr}(t,\cdot)|^{(1)}_{|\gamma|_+}
+|c_{kr}(t,\cdot)|^{(2)}_{|\gamma|_+} +
|\sigma^{i}_{kr}(t,\cdot)|^{(0)}_{|\gamma+1|_+} +
|\nu_{kr}(t,\cdot)|^{(1)}_{|\gamma+1|_+} \leq L.
$$
\end{assumption}

\begin{remark}
          \label{05.18.01}
(i). The condition (\ref{12.10.1}) is equivalent to
\begin{eqnarray}
\lim_{\rho(x)\to 0}\sup_{\omega,t}
\left(osc(a^{ij}_{kr})_{B_{\frac{\rho(x)}{2}}(x)}+osc(\sigma^i_{kr})_{B_{\frac{\rho(x)}{2}}(x)}\right)=0.\nonumber
\end{eqnarray}

(ii).  It is easy to see
 that (\ref{12.10.1}) is much weaker than uniform continuity condition.
   For instance,
if $\delta\in(0,1)$, $d=d_1=1$, and $\cO=\bR_{+}$, then the function
$a(x)$ equal to $2+\sin (|\ln x|^{\delta})$ for $0<x\leq1/2$
satisfies (\ref{12.10.1}). Indeed, if $x,y>0$ and $ |x-y|\leq
x\wedge y $, then
 $$
|a(x)-a(y)|=|x-y||a'(\xi)|,
$$
where $\xi$ lies between $x$ and $y$. In addition, $|x-y|\leq
x\wedge y \leq\xi\leq2(x\wedge y)$,  and $ \xi|a'(\xi)|\leq
|\ln[2(x\wedge y)]|^{\delta-1}\to0$ as $x\wedge y\to0$.

(iii). We observe that (\ref{05.04.01}) allows the coefficients
$b^i_{kr}, c_{kr}$ and $\nu_{kr}$ to blow up near the boundary at a
certain rate. For instance, it holds if
\begin{eqnarray}
|b^i_{kr}|\le N \rho^{-1+\varepsilon},\quad |c_{kr}|\le N
\rho^{-2+\varepsilon},\quad |\nu_{kr}|_{\ell_2}\le N
\rho^{-1+\varepsilon}\nonumber
\end{eqnarray}
for some constants $N$, $\varepsilon>0$.
\end{remark}

Here is the main result of this article.

\begin{theorem}
                    \label{main theorem on domain}
Let Assumptions \ref{main assumptions}, \ref{assump 9.17} and
\ref{assumption regularity} hold. Then for any $f\in
\psi^{-1}\bH^{\gamma}_{2,\theta}(\cO,T),\;g\in
\bH^{\gamma+1}_{2,\theta}(\cO,T,\ell_2),\; u_0\in
U^{\gamma+2}_{2,\theta}(\cO)$,  the problem (\ref{eqn main system})
on $\Omega\times[0,T]\times\cO$ admits a unique solution
$u=(u^1,\cdots, u^{d_1})\in \frH^{\gamma+2}_{2,\theta}(\cO,T)$, and
for this solution
\begin{equation}
                        \label{a priori}
\|\psi^{-1}u\|_{\bH^{\gamma+2}_{2,\theta}(\cO,T)}\leq c\left(\|\psi
f\|_{\bH^{\gamma}_{2,\theta}(\cO,T)}
+\|g\|_{\bH^{\gamma+1}_{2,\theta}(\cO,T,\ell_2)}
+\|u_0\|_{U^{\gamma+2}_{2,\theta}(\cO)}\right),
\end{equation}
where $c=c(d,d_1,\delta,\theta,K,L,T)$.
\end{theorem}

\begin{remark}
By inspecting the proofs carefully, one can check that
 the above theorem hold true even if $\cO$ is not bounded.

\end{remark}

\begin{proof}

Since the theorem was already proved for single equations
(\cite{Kim03}), as in the proof of Theorem \ref{thm 1}, we only need
to  establish the a priori estimate (\ref{a priori}) assuming that a
solution $u\in \frH^{\gamma+2}_{2,\theta}(\cO,T)$ already exists. As
usual, we assume  $u_0=0$.

 Let $x_0 \in \partial \cO$ and $\Psi$ be a function from
Assumption \ref{assumption domain}. In \cite{KK2} it is shown that
 $\Psi$ can be chosen in such a way that  for any non-negative integer $n$
\begin{equation}
                                                          \label{2.25.03}
|\Psi_{x}|^{(0)}_{n,B_{r_0}(x_0)\cap \cO} +
 |\Psi^{-1}_{x}|^{(0)}_{n,J_{+}} < N(n)<  \infty
\end{equation}
and
\begin{equation}
                                                     \label{2.25.02}
\rho(x)\Psi_{xx}(x) \to 0 \quad \text{as}\quad  x\in
B_{r_0}(x_0)\cap \cO,
 \text{and} \,\,\,  \rho(x) \to 0,
\end{equation}
where the constants $N(n)$ and the
 convergence in (\ref{2.25.02}) are independent of  $x_0$.

 Define $r=r_{0}/K_{0}$
and fix smooth functions $\eta \in C^{\infty}_{0}(B_1(0) ),
\varphi\in C^{\infty}(\bR)$ such that $ 0 \leq \eta, \varphi \leq
1$, and
 $\eta=1$ in $B_{r/2}(0) $,  $\varphi(t)=1$ for $t\leq -3$, and
$\varphi(t)=0$
 for $t\geq-1$ and  $0\geq\varphi'\geq-1$. Observe that
$\Psi(B_{r_0}(x_0))$ contains $B_r(0) $.
 For $n=1,2,... $, $t>0$ and $x\in\bR^{d}_{+}$ we introduce
$\varphi_{n}(x)=\varphi(n^{-1}\ln x^1)$,
$$
a^{ij,n}=(a^{ij,n}_{kr}):= \tilde{a}^{ij}  \eta(x)\varphi_{n} +
\delta^{ij}(1- \eta \varphi_{n})I, \quad b^{i,n}=(b^{i,n}_{kr})
:=\tilde{b}^i \eta \varphi_{n},\quad c^{n}=(c^n_{kr}) :=\tilde{c}
\eta \varphi_{n},
$$
$$
\sigma^{i,n}=(\sigma^{i,n}_{kr}):=\tilde{\sigma}\eta\varphi_{n},\quad
\nu^n=(\nu^{n}_{kr}):=\tilde{\nu}\eta \varphi_{n},
$$
where
$$
\tilde{a}^{ij}(t,x)=\bar{a}^{ij}(t,\Psi^{-1}(x)), \quad
\tilde{b}^{i}(t,x)=\bar{b}^{i}(t,\Psi^{-1}(x)),
$$
$$
\tilde{\sigma}^{i}(t,x)=\bar{\sigma}^{i}(t,\Psi^{-1}(x)),
 \quad \bar{a}^{ij}=\sum_{s,t=1}^d a^{st}\Psi^{i}_{x^{s}}\Psi^{j}_{x^{t}},
$$
$$
\quad\bar{b}^{i}=\sum_{s,t}a^{st}\Psi^{i}_{x^{s}x^{t}}
+\sum_{\ell}b^{\ell}\Psi^{i}_{x^{\ell}},
\quad\bar{\sigma}^{i}=\sum_{s}\sigma^{s}\Psi^{i}_{x^s},
$$
$$
\tilde{c}(t,x)=c(t,\Psi^{-1}(x)),\quad
\tilde{\nu}(t,x)=\nu(t,\Psi^{-1}(x)).
$$
Using  Lemma 3.4 of \cite{KK2}, one can easily
  check that there is a constant $L'$ independent of $n$ and $x_0$
  so that
\begin{equation}
                     \label{eqn 9.21.1}
|a^{ij,n}(t,\cdot)|^{(0)}_{|\gamma|_+}
+|b^{i,n}(t,\cdot)|^{(1)}_{|\gamma|_+}
+|c^n(t,\cdot)|^{(2)}_{|\gamma|_+} +
|\sigma^{i,n}_{kr}(t,\cdot)|^{(0)}_{|\gamma+1|_+} +
|\nu^{n}_{kr}(t,\cdot)|^{(1)}_{|\gamma+1|_+} \leq L'.
\end{equation}
Take  $\kappa_0$ from Theorem \ref{theorem half} corresponding to
$d,\theta,\delta, K, \gamma$ and $L'$.
 Observe that $\varphi(m^{-1} \ln x^1)=0$ for $x^1 \geq e^{-m}$ and
 $|\varphi(m^{-1}\ln x^1) - \varphi(m^{-1} \ln y^1)| \leq m^{-1}$
 if $|x^1-y^1|\leq x^1 \wedge y^1$.
 Also we easily see that
 (\ref{2.25.02}) implies $x^{1}\Psi_{xx}(\Psi^{-1}(x)) \to 0$ as $x^1\to 0$.
  Using these facts, (\ref{12.10.1}) and (\ref{05.04.01}),
 one can find and fix
 $n>0$ independent of $x_0$ such that
$$
|a^{ij,n}_{kr}(t,x)-a^{ij,n}_{kr}(t,y)| + |\sigma^{i,n}_{kr}(t,x)
-\sigma^{i,n}_{kr}(t,y)|_{\ell_2} + x^1 |b^{i,n}_{kr}(t,x)|
$$
\begin{equation}
             \label{eqn 9.21.2}
+ (x^1)^2|c^{n}_{kr}(t,x)|+ x^1|\nu^n_{kr}(t,x)|_{\ell_2} \leq
\kappa_0,
\end{equation}
whenever $t>0, x,y\in \bR^d_+$ and $|x-y|\leq x^1 \wedge y^1$.
  Now we fix  a $\rho_0  <r_0  $ such that
\begin{equation}
                    \label{eqn 8.21.4}
\Psi(B_{\rho_0}(x_0)) \subset B_{r/2}(0) \cap \{x:x^1 \leq
e^{-3n}\}.
\end{equation}
 Let $\xi$ be a smooth function
with support in $B_{\rho_0}(x_0)$ and denote $v:=(\xi u)(\Psi^{-1})$
and continue $v$ as zero in
$\bR^{d}_{+}\setminus\Psi(B_{\rho_0}(x_0))$. Since
$\eta\varphi_{n}=1$ on $\Psi(B_{\rho_0}(x_0))$, the function  $v$
satisfies
$$
d v^k = (a^{ij,n}_{kr}v^r_{x^i x^j} + b^{i,n}_{kr}v^r_{x^i} +
c^n_{kr}v + \hat{f}^k) \,dt + (\sigma^{i,n}_{kr,m}v^r_{x^i} +
\nu^{n}_{kr,m}v + \hat{g}^k_m)\, dw^m_t,
$$
where
$$\hat{f}^k =\tilde{f}(\Psi^{-1}), \quad \tilde{f}=
-2a^{ij}_{kr}u^r_{x^{i}}\xi_{x^{j}}
-a^{ij}_{kr}u^r\xi_{x^{i}x^{j}}-b^{i}_{kr}u^r\xi_{x^{i}} +\xi f^k,
$$
$$
\hat{g}=\tilde{g}(\Psi^{-1}), \quad \tilde{g}^k=-\sigma^{i}_{kr}u^r
\xi_{x^i}+\xi g^k.
$$
Next, we observe  that by Lemma \ref{lemma 10.3.1} and
  Theorem 3.2 in \cite{Lo2} (or see \cite{KK2})
for any $\nu,\alpha \in \bR $ and $h \in
\psi^{-\alpha}H^{\nu}_{p,\theta}(\cO)$ with support in
$B_{\rho_0}(x_0)$ we have
\begin{equation}
                                                          \label{1.28.01}
\|\psi^{\alpha}h\|_{H^{\nu}_{p,\theta}(\cO)} \sim
\|M^{\alpha}h(\Psi^{-1})\|_{H^{\nu}_{p,\theta}}.
\end{equation}
Therefore, we  conclude that
 $v\in \frH^{\gamma+2}_{2,\theta}(T)$.
Hence, by Theorem \ref{theorem half}  we get for any $t\leq T$
$$
\|M^{-1}v\|_{\bH^{\gamma+2}_{2,\theta}(t)} \leq N \left(\|M\hat{f}
\|_{\bH^{\gamma}_{2,\theta}(t)} +
\|\hat{g}\|_{\bH^{\gamma+1}_{2,\theta}(t,\ell_2)} +
 \|u_0(\Psi^{-1}) \zeta( \Psi^{-1})\|_{U^{\gamma+2}_{2,\theta}}\right).
$$
By using (\ref{1.28.01}) again, we obtain
\begin{eqnarray}
&&\|\psi^{-1}u\zeta\|_{\bH^{\gamma+2}_{2,\theta}(\cO,t)}\nonumber\\
 &\leq& N
\|a\xi_x \psi u_x\|_{\bH^{\gamma}_{2,\theta}(\cO,t)} + N
\|a\xi_{xx}\psi u\|_{\bH^{\gamma}_{2,\theta}(\cO,t)} + N \|\xi_x
\psi b u\|_{\bH^{\gamma}_{2,\theta}(\cO,t)}\nonumber\\
&&+ N \|\sigma \xi_x u\|_{\bH^{\gamma+1}_{2,\theta}(\cO,t)} +  N
\|\xi \psi f\|_{\bH^{\gamma}_{2,\theta}(\cO,t)} + \|\xi
g\|_{\bH^{\gamma+1}_{2,\theta}(\cO,t,\ell_2)} +\|\xi u_0
\|_{U^{\gamma+2}_{2,\theta}(\cO)}.\nonumber
\end{eqnarray}
Remembering   that
 $\rho$ and $\psi$ are comparable in $\cO$, one can easily check
   that the functions
$$
|\xi_x a(t,\cdot)|^{(0)}_{|\gamma|_+},\,\, |\xi_{xx}\psi
a(t,\cdot)|^{(0)}_{|\gamma|_+},\,\, |\xi_x \psi
b(t,\cdot)|^{(0)}_{|\gamma|_+},\,\, | \xi_x
\sigma(t,\cdot)|^{(0)}_{|\gamma+1|_+}
$$
are bounded on $\Omega \times [0,T]$. Then one   concludes

\begin{eqnarray}
&&\|\psi^{-1}u\xi\|_{\bH^{\gamma+2}_{2,\theta}(\cO,t)} \nonumber\\
&\leq& N \|\psi u_x\|_{\bH^{\gamma}_{2,\theta}(\cO,t)} + N
\|u\|_{\bH^{\gamma}_{2,\theta}(\cO,t)}+N \|\psi
f\|_{\bH^{\gamma}_{2,\theta}(\cO,t)} + \|
g\|_{\bH^{\gamma+1}_{2,\theta}(\cO,t,\ell_2)} +N
\|u_0\|_{U^{\gamma+2}_{2,\theta}(\cO)}.\nonumber
\end{eqnarray}

Note that the above constants $\rho_0, m, L', N$  are independent of
$x_0$.  Therefore, to estimate the norm
 $\|\psi^{-1} u\|_{\bH^{\gamma+2}_{2,\theta}(\cO,t)}$,
 one introduces a partition of unity $\xi_{(i)}, i=0,1,2,...,N$ such
that $\xi_{(0)} \in C^{\infty}_0(\cO)$ and
  $\xi_{(i)} \in C^{\infty}_0(B_{\rho_0}(x_i))$,
$ x_i \in \partial \cO$ for $i\geq1$.
 Then one estimates
$\|\psi^{-1} u\xi_{(0)}\|_{\bH^{\gamma+2}_{2,\theta}(\cO,t)}$ using
Theorem \ref{thm 2} and the other norms as above.  We only mention
that since $\psi^{-1} u\xi_{(0)}$ has compact support in $\cO$,
$$
\|\psi^{-1} u\xi_{(0)}\|_{\bH^{\gamma+2}_{2,\theta}(\cO,t)}\sim \|
u\xi_{(0)}\|_{\bH^{\gamma+2}_{2,\theta}(\cO,t)}\sim \|
u\xi_{(0)}\|_{\bH^{\gamma+2}_{2}(\bR^d,t)}.
$$
By summing up those estimates one gets
\begin{eqnarray}
&&\|\psi^{-1} u\|_{\bH^{\gamma+2}_{2,\theta}(\cO,t)}\nonumber\\
 &\leq& N
 \|\psi u_x\|_{\bH^{\gamma}_{2,\theta}(\cO,t)}+
N\|u\|_{\bH^{\gamma}_{2,\theta}(\cO,t)}+ N \|\psi
f\|_{\bH^{\gamma}_{2,\theta}(\cO,t)} + N
\|g\|_{\bH^{\gamma+1}_{p,\theta}(\cO,t,\ell_2)} + N
\|u_0\|_{U^{\gamma+2}_{2,\theta}(\cO)}.\nonumber
\end{eqnarray}
By this  and the inequality
$$
\|\psi u_x\|_{ H^{\gamma}_{2,\theta}(\cO)} \leq N \|u\|_{
H^{\gamma+1}_{2,\theta}(\cO)},
$$
  we get for each $t\leq T$,
\begin{equation}
                        \label{eqn 9.21.6}
\|u\|^2_{\frH^{\gamma+2}_{2,\theta}(\cO,t)} \leq N
\|u\|^2_{\bH^{\gamma+1}_{2,\theta}(\cO,t)} + N \left(\|\psi
f\|^2_{\bH^{\gamma}_{2,\theta}(\cO,T)} +
\|g\|^2_{\bH^{\gamma+1}_{2,\theta}(\cO,T,\ell_2)} +
\|u_0\|^2_{U^{\gamma+2}_{2,\theta}(\cO)}\right).
\end{equation}
 Now the a priori estimate  follows from  Lemma \ref{lemma 15.05}
 and Gronwall's inequality.
 The theorem is proved.
\end{proof}

\end{document}